\documentclass[12pt]{amsart}
\usepackage{amscd,amssymb}
\usepackage{amsthm,amsmath,amssymb}
\usepackage[matrix,arrow,curve]{xy}
\usepackage{mathrsfs}
\usepackage{supertabular,longtable, rotating}

\newtheorem{theorem}[equation]{Theorem}

\newtheorem{lemma}[equation]{Lemma}
\newtheorem{corollary}[equation]{Corollary}

\newtheorem{problem}[equation]{Problem}

\theoremstyle{definition}

\newtheorem{definition}[equation]{Definition}

\theoremstyle{remark}
\newtheorem{remark}[equation]{Remark}

\makeatletter\@addtoreset{equation}{section} \makeatother

\def\P {\mathbb{P}}

\def\SS {\mathfrak{S}}
\def\A {\mathfrak{A}}
\def\mumu{\boldsymbol{\mu}}
\def\vstrut {\vphantom{$\sqrt{0}^A_j$}}

\author{Ivan Cheltsov and Constantin Shramov}

\title{Two rational nodal quartic threefolds}

\address{\emph{Ivan Cheltsov}
\newline
\textnormal{School of Mathematics, The University of Edinburgh,  Edinburgh EH9 3JZ, UK.}
\newline
\textnormal{National Research University Higher School of Economics, Russian Federation, AG Laboratory, HSE, 7 Vavilova str., Moscow, 117312, Russia.}
\newline
\textnormal{\texttt{I.Cheltsov@ed.ac.uk}}}

\address{\emph{Constantin Shramov}
\newline
\textnormal{Steklov Institute of Mathematics, 8 Gubkina street, Moscow 119991, Russia.}
\newline
\textnormal{National Research University Higher School of Economics, Russian Federation, AG Laboratory, HSE, 7 Vavilova str., Moscow, 117312, Russia.}
\newline
\textnormal{\texttt{costya.shramov@gmail.com}}}

\begin{document}

\begin{abstract}
We prove that the quartic threefolds defined by
$$
\sum_{i=0}^{5}x_i=\sum_{i=0}^{5}x_i^4-t\left(\sum_{i=0}^{5}x_i^2\right)^2=0
$$
in $\mathbb{P}^5$ are rational for $t=\frac{1}{6}$ and $t=\frac{7}{10}$.
\end{abstract}

\sloppy

\maketitle

\section{Introduction}
\label{section:intro}

Consider the six-dimensional permutation representation $\mathbb{W}$
of the group~$\mathfrak{S}_6$. Choose coordinates $x_0,\ldots,x_5$
in $\mathbb{W}$ so that they are permuted by~$\mathfrak{S}_6$.
Then~\mbox{$x_0\ldots,x_5$} also serve as homogeneous coordinates
in the projective space~\mbox{$\mathbb{P}^5=\mathbb{P}(\mathbb{W})$}.

Let us identify $\mathbb{P}^4$ with a hyperplane
$$x_0+x_1+x_2+x_3+x_4+x_5=0$$
in $\mathbb{P}^5$.
Denote by $X_{t}$ the quartic threefold in $\mathbb{P}^4$ that is given by the equation
\begin{equation}
\label{equation:quartic}
x_0^4+x_1^4+x_2^4+x_3^4+x_4^4+x_5^4=t\Big(x_0^2+x_1^2+x_2^2+x_3^2+x_4^2+x_5^2\Big)^2,
\end{equation}
where $t$ is an element of the ground field, which we will
always assume to be the field~$\mathbb{C}$ of complex numbers.
Every $\SS_6$-invariant quartic in $\P(\mathbb{W})$
is one of the quartics~$X_t$. Moreover, every quartic
threefold with a faithful $\SS_6$-action is isomorphic to some~$X_t$.
All quartics~$X_t$ are singular.
Indeed, denote by $\Sigma_{30}$ the $\mathfrak{S}_6$-orbit of
the point~\mbox{$[1:1:\omega:\omega:\omega^2:\omega^2]$},
where~\mbox{$\omega=e^{\frac{2\pi i}{3}}$}.
Then $|\Sigma_{30}|=30$, and $X_{t}$ is singular at every point of $\Sigma_{30}$ for every $t\in\mathbb{C}$ (see, for example,  \cite[Theorem~4.1]{Geer}).

The possible singularities of the quartic threefold $X_t$ have been described by van der Geer in \cite[Theorem~4.1]{Geer}.
To recall his description, denote by $\mathcal{L}_{15}$ the $\mathfrak{S}_6$-orbit of the line
that passes through the points $[1:0:-1:1:0:-1]$ and $[0:1:-1:0:1:-1]$,
and denote by $\Sigma_{6}$, $\Sigma_{10}$, and $\Sigma_{15}$ the $\mathfrak{S}_6$-orbits of the points
$[-5:1:1:1:1:1]$, $[-1:-1:-1:1:1:1]$, and $[1:-1:0:0:0:0]$, respectively.
Then the curve $\mathcal{L}_{15}$ is a union of fifteen lines, while
$|\Sigma_{6}|=6$, $|\Sigma_{10}|=10$, and $|\Sigma_{15}|=15$.
Moreover, one has
$$
\mathrm{Sing}(X_t)=\left\{\aligned
&\mathcal{L}_{15}\ \text{if}\ t=\frac{1}{4},\\
&\Sigma_{30}\cup\Sigma_{15}\ \text{if}\ t=\frac{1}{2},\\
&\Sigma_{30}\cup\Sigma_{10}\ \text{if}\ t=\frac{1}{6},\\
&\Sigma_{30}\cup\Sigma_{6}\ \text{if}\ t=\frac{7}{10},\\
&\Sigma_{30}\ \text{otherwise}.\\
\endaligned
\right.
$$
Furthermore, if $t\ne\frac{1}{4}$, then all singular points of the quartic threefold $X_t$ are isolated ordinary double points (nodes).

The threefold $X_{\frac{1}{2}}$ is classical.
It is the so-called \emph{Burkhardt quartic}.
In \cite{Burkhardt}, Burkhardt discovered that the subset
$\Sigma_{30}\cup\Sigma_{15}$ is
invariant under the action of the simple group~\mbox{$\mathrm{PSp}_4(\mathbf{F}_3)$} of order~$25920$.
In \cite{Coble}, Coble proved that $\Sigma_{30}\cup\Sigma_{15}$ is the singular locus
of the threefold~$X_{\frac{1}{2}}$,
and proved that $X_{\frac{1}{2}}$ is also $\mathrm{PSp}_4(\mathbf{F}_3)$-invariant.
Later Todd proved in~\cite{To36a} that~$X_{\frac{1}{2}}$ is rational.
In \cite{JSV1990}, de Jong, Shepherd-Barron, and Van de Ven proved
that~$X_{\frac{1}{2}}$ is the unique quartic threefold in $\mathbb{P}^4$ with $45$ singular points.

The quartic threefold $X_{\frac{1}{4}}$ is also classical.
It is known as the \emph{Igusa quartic} from its modular interpretation as
the Satake compactification of the moduli space of Abelian surfaces with
level $2$ structure (see \cite{Geer}).
The projectively dual variety of the quartic threefold $X_{\frac{1}{4}}$ is the so-called \emph{Segre cubic}.
Since the Segre cubic is rational, $X_{\frac{1}{4}}$ is rational as well.

During \emph{Kul!fest} conference dedicated to the $60$th anniversary of Viktor Kulikov
that was held in Moscow in December 2012, Alexei Bondal and Yuri Prokhorov posed

\begin{problem}
\label{problem:Bondal-Prokhorov}
Determine all $t\in\mathbb{C}$ such that $X_{t}$ is rational.
\end{problem}

Since $X_t$ is singular, we cannot apply Iskovskikh and Manin's theorem
from \cite{IM71} to~$X_t$.
Similarly, we cannot apply Mella's \cite[Theorem~2]{Mella} to $X_t$ either,
because the quartic threefold $X_t$ is not $\mathbb{Q}$-factorial by \cite[Lemma~2]{Beauville}.
Nevertheless, Beauville proved

\begin{theorem}[{\cite{Beauville}}]
\label{theorem:Beauville}
If $t\not\in\{\frac{1}{2},\frac{1}{4},\frac{1}{6},\frac{7}{10}\}$, then $X_{t}$ is non-rational.
\end{theorem}

Both $X_{\frac{1}{2}}$ and $X_{\frac{1}{4}}$ are rational. The goal of this paper is to prove

\begin{theorem}
\label{theorem:Todd}
The quartic threefolds $X_{\frac{1}{6}}$ and $X_{\frac{7}{10}}$ are also rational.
\end{theorem}

Surprisingly, the proof of Theorem~\ref{theorem:Todd} goes back to two classical papers of Todd.
Namely, we will construct an explicit $\mathfrak{A}_6$-birational map $\mathbb{P}^3\dasharrow X_{\frac{7}{10}}$
that is a special case of Todd's construction from~\cite{To33}.
Similarly, we will construct an explicit $\mathfrak{S}_5$-birational map $\mathbb{P}^3\dasharrow X_{\frac{1}{6}}$
that is a degeneration of Todd's construction from~\cite{To35}.
We emphasize that our proof is self-contained, i.e. it does not rely on
the results proved in~\cite{To33} and~\cite{To35}, but recovers the necessary
facts in our particular situation using additional symmetries arising from
group actions.

\begin{remark}
\label{remark:determinantal}
Todd proved in \cite{To36a} that the Burkhardt quartic $X_{\frac{1}{2}}$ is determinantal (see also~\mbox{\cite[\S5.1]{Pet98}}).
The constructions of our birational maps $\mathbb{P}^3\dasharrow X_{\frac{7}{10}}$ and $\mathbb{P}^3\dasharrow X_{\frac{1}{6}}$
imply that both $X_{\frac{7}{10}}$ and $X_{\frac{1}{6}}$ are determinantal (see \cite[Example~6.4.2]{Pet98} and \cite[Example~6.2.1]{Pet98}).
Yuri Prokhorov pointed out that the quartic threefold
$$
\mathrm{det}\begin{pmatrix}
y_0& y_1& y_2& y_3\\
y_4& y_0& y_3& y_4\\
y_2& y_1& y_1& y_0\\
y_0& y_3& y_2& y_4
\end{pmatrix}=0
$$
in $\mathbb{P}^4$ with homogeneous coordinates $y_0,\ldots,y_4$
has exactly $45$ singular points. Thus, it is isomorphic to the Burkhardt quartic $X_{\frac{1}{2}}$ by \cite{JSV1990}.
It would be interesting to find similar determinantal equations of the threefolds $X_{\frac{7}{10}}$ and $X_{\frac{1}{6}}$.
\end{remark}

The plan of the paper is as follows. In Section~\ref{section:representations}
we recall some preliminary results on representations
of a central extension of the group $\mathfrak{S}_6$, and some of its
subgroups. In Section~\ref{section:icosahedron} we collect
results concerning a certain action of the group $\mathfrak{A}_5$ on
$\mathbb{P}^3$, and study $\mathfrak{A}_5$-invariant quartic surfaces; the reason we pay so much
attention to this group is that it is contained both in $\mathfrak{A}_6$ and in $\mathfrak{S}_5$,
and thus the information about its properties simplifies the study of the latter
two groups.
In Section~\ref{section:basic-actions} we collect auxiliary results
about the groups~$\mathfrak{S}_6$, $\mathfrak{A}_6$ and~$\mathfrak{S}_5$, in particular about their
actions on curves and their five-dimensional irreducible representations.
In Section~\ref{section:7-10} we construct
an $\mathfrak{A}_6$-equivariant birational
map~\mbox{$\mathbb{P}^3\dasharrow X_{\frac{7}{10}}$}.
Finally, in Section~\ref{section:1-6}
we construct an $\mathfrak{S}_5$-equivariant birational
map~\mbox{$\mathbb{P}^3\dasharrow X_{\frac{1}{6}}$}
and make some concluding remarks.

Throughout the paper, we denote a cyclic group of
order $n$ by $\mumu_n$, and we denote a dihedral group
of order $2n$ by $\mathrm{D}_{2n}$. In particular, one
has $\mathrm{D}_{12}\cong\mathfrak{S}_3\times\mumu_2$.
By $F_{36}$ we denote a group isomorphic
to~\mbox{$(\mumu_3\times\mumu_3)\rtimes \mumu_4$}, and
by $F_{20}$ we denote a group isomorphic to~\mbox{$\mumu_5\rtimes \mumu_4$}.

\medskip

{\bf Acknowledgements.}
This work was supported within the framework of a subsidy granted to the HSE
by the Government of the Russian Federation for the implementation of the Global Competitiveness Program.
Constantin Shramov was also supported by the grants
RFFI 15-01-02158, RFFI 15-01-02164, RFFI 14-01-00160, MK-2858.2014.1,
and NSh-2998.2014.1, and by Dynasty foundation.

We are grateful to Dmitrii Pasechnik, Yuri Prokhorov, Kristian Ranestad,
Leonid Rybnikov, and Andrey Trepalin
for very helpful and useful discussions.

\section{Representation theory}
\label{section:representations}

Recall that the permutation group $\mathfrak{S}_6$ has two central
extensions $2^+\mathfrak{S}_6$ and $2^-\mathfrak{S}_6$ by the group~$\mumu_2$
with the central subgroup contained
in the commutator subgroup (see~\mbox{\cite[p.~xxiii]{Atlas}} for details).
We denote the first of them (i.\,e. the one where the preimages
of a transposition in $\mathfrak{S}_6$ under the natural projection
have order two) by $2.\mathfrak{S}_6$ to simplify notation.
Similarly, for any group $\Gamma$ we denote by $2.\Gamma$ a non-split
central extension of $\Gamma$ by the group~$\mumu_2$.

We start with recalling some facts about four- and five-dimensional
representations
of the group~$2.\mathfrak{S}_6$ we will be working with.
A reader who is not interested in details here can skip to Corollary~\ref{corollary:characters},
or even to Section~\ref{section:basic-actions} where we reformulate everything
in geometric language. Also, we will see in Section~\ref{section:basic-actions} that our further constructions do not depend much on the choice of representations, and all computations one makes for one of them
actually apply to all others.

Let $\mathbb{I}$ and $\mathbb{J}$ be the trivial and the non-trivial one-dimensional
representations of the group $\mathfrak{S}_6$, respectively.
Consider the six-dimensional permutation representation $\mathbb{W}$
of~$\mathfrak{S}_6$. One has
$$
\mathbb{W}\cong \mathbb{I}\oplus \mathbb{W}_5\otimes \mathbb{J}
$$
for some irreducible representation $\mathbb{W}_5$ of $\mathfrak{S}_6$.
We can regard $\mathbb{I}$, $\mathbb{J}$ and $\mathbb{W}_5$ as representations of
the group $2.\mathfrak{S}_6$.
Recall that there is a double cover
$$\nabla\colon\mathrm{SL}_4(\mathbb{C})\to\mathrm{SO}_6(\mathbb{C}),$$
see e.g.~\mbox{\cite[Exercise~20.39]{FultonHarris}}.
Taking the embedding of the group $2.\SS_6$ into $\mathrm{SO}_6(\mathbb{C})$
via the representation $\mathbb{I}\oplus\mathbb{W}_5$ and considering
its preimage with respect to $\nabla$, we produce
an embedding of the group~$2.\mathfrak{S}_6$
into~$\mathrm{SL}_4(\mathbb{C})$.
This embedding gives rise to two four-dimensional
representations of~$2.\mathfrak{S}_6$ that differ by a tensor product
with~$\mathbb{J}$.
We fix one of these two representations~$\mathbb{U}_4$.
Note that
$$
\mathbb{I}\oplus\mathbb{W}_5\cong \Lambda^2(\mathbb{U}_4).
$$

Recall that there are coordinates $x_0,\ldots,x_5$ in $\mathbb{W}$
that are permuted by the group $\mathfrak{S}_6$.
We will refer to a subgroup
of~\mbox{$2.\mathfrak{S}_6$} fixing one of the corresponding points
as a \emph{standard} subgroup $2.\mathfrak{S}_5$; we denote
any such subgroup by $2.\mathfrak{S}_5^{st}$. A subgroup
of $2.\mathfrak{S}_6$ that is isomorphic to $2.\mathfrak{S}_5$ but is not conjugate
to a standard~\mbox{$2.\mathfrak{S}_5$} will be called a \emph{non-standard}
subgroup $2.\mathfrak{S}_5$; we denote any such subgroup by $2.\mathfrak{S}_5^{nst}$.
These agree with standard and non-standard
subgroups of $\mathfrak{S}_6$ isomorphic to~$\mathfrak{S}_5$, although outer
automorphisms of $\mathfrak{S}_6$ do not lift to $2.\mathfrak{S}_6$.
Any subgroup of $2.\mathfrak{S}_6$ that is isomorphic to~$2.\mathfrak{A}_5$,
$2.\mathfrak{S}_4$ or~$2.\mathfrak{A}_4$ and is contained in $2.\mathfrak{S}_5^{st}$
is denoted by $2.\mathfrak{A}_5^{st}$, $2.\mathfrak{S}_4^{st}$ or $2.\mathfrak{A}_4^{st}$, respectively.
Similarly, any subgroup of $2.\mathfrak{S}_6$ that is isomorphic to
$2.\mathfrak{A}_5$, $2.\mathfrak{S}_4$ or $2.\mathfrak{A}_4$ and is contained in $2.\mathfrak{S}_5^{nst}$
is denoted by $2.\mathfrak{A}_5^{nst}$, $2.\mathfrak{S}_4^{nst}$ or $2.\mathfrak{A}_4^{nst}$, respectively.

The values of characters of important representations of the group $2.\mathfrak{S}_6$,
and the information about some of its subgroups are presented
in Table~\ref{table:characters}, cf.~\cite[p.~5]{Atlas}.
The first two columns of Table~\ref{table:characters} describe conjugacy classes
of elements of the group $2.\mathfrak{S}_6$. The first column lists the orders of
the elements in the corresponding conjugacy class, and the second column, except for the
entries in the second and the third row,
gives a cycle type of the image of an element under projection to $\mathfrak{S}_6$
(for example, $[3,2]$ denotes a product of two disjoint cycles of
lengths $3$ and $2$). By $\mathrm{id}$ we denote the identity element
of $2.\mathfrak{S}_6$, and $z$ denotes the unique non-trivial central element
of $2.\mathfrak{S}_6$.
Note that the preimages of some of conjugacy classes in $\mathfrak{S}_6$ split into a union
of two conjugacy classes in $2.\mathfrak{S}_6$.
The next three columns list the values of
the characters of the representations
$\mathbb{W}$, $\mathbb{W}_5$
and $\mathbb{U}_4$ of $2.\mathfrak{S}_6$.
Note that there is no real ambiguity in the choice of $\sqrt{-3}$
since we did not specify any way to distinguish the two conjugacy classes
in $2.\mathfrak{S}_6$ whose elements are projected to cycles of length $6$ in $\mathfrak{S}_6$
up to this point (note that the two ways to choose a sign here is exactly
a tensor multiplication of the representation with~$\mathbb{J}$,
i.e. the choice between two homomorphisms of $2.\mathfrak{S}_6$ to
$\mathrm{SL}_4(\mathbb{C})$ having the same image).
The remaining columns list the numbers of elements from each of the conjugacy classes
of $2.\mathfrak{S}_6$ in subgroups of certain types.
By~\mbox{$2.F_{36}$} (respectively, by $2.F_{20}$,
by~\mbox{$2.\mathrm{D}_{12}^{nst}$}) we denote a subgroup of $2.\mathfrak{S}_6$
(respectively, of~\mbox{$2.\mathfrak{S}_6$}, or of $2.\mathfrak{S}_5^{nst}$)
isomorphic to a central extension of
$F_{36}$ (respectively, of $F_{20}$, or of $\mathrm{D}_{12}$) by~$\mumu_2$.
A subgroup $2.F_{20}$ is actually contained in a subgroup
$2.\mathfrak{S}_5^{st}$ and in a subgroup~\mbox{$2.\mathfrak{S}_5^{nst}$}.
Note that the intersection of a conjugacy class
in a group with a subgroup may (and often does) split
into several conjugacy classes in this subgroup.

\begin{table}[p]
\setbox1=\hbox{\parbox{1.5\linewidth}{%
\begin{minipage}{\textheight}
\centering
\caption{Characters and subgroups of the group $2.\mathfrak{S}_6$\label{table:characters}}
\begin{tabular}{|c|c||c|c|c||c|c|c|c|c|c|c|c|c|c|}
\hline
$\mathrm{ord}$& type &
$\mathbb{W}$ &
$\mathbb{W}_5$ &
$\mathbb{U}_4$ &
$2.\mathfrak{S}_6$ & $2.\mathfrak{A}_6$
& $2.\mathfrak{S}_5^{nst}$
& $2.\mathfrak{A}_5^{st}$ & $2.\mathfrak{A}_5^{nst}$
& $2.\mathfrak{S}_4^{nst}$
& $2.\mathfrak{A}_4^{nst}$
& $2.F_{36}$
& $2.F_{20}$
& $2.\mathrm{D}_{12}^{nst}$
\vstrut \\
\hline
$1$ & $\mathrm{id}$ & $6$ & $5$ & $4$ & $1$ & $1$ & $1$ & $1$ & $1$ & $1$ & $1$ & $1$ & $1$ & $1$
\vstrut \\
\hline
$2$ & $z$ & $6$ & $5$ & $-4$ & $1$ & $1$ & $1$ & $1$ & $1$ & $1$ & $1$ & $1$ & $1$ & $1$
\vstrut \\
\hline
$2$ & $[2]$ & $4$ & $-3$ & $0$ & $30$ & $0$ & $0$ & $0$ & $0$ & $0$ & $0$ & $0$ & $0$ & $0$
\vstrut \\
\hline
$4$ & $[2,2]$ & $2$ & $1$ & $0$ & $90$ & $90$ & $30$ & $30$ & $30$ & $6$ & $6$ & $18$ & $10$ & $6$
\vstrut \\
\hline
$4$ & $[2,2,2]$ & $0$ & $1$ & $0$ & $30$ & $0$ & $20$ & $0$ & $0$ & $12$ & $0$ & $0$ & $0$ & $8$
\vstrut \\
\hline
$6$ & $[3]$ & $3$ & $2$ & $2$ & $40$ & $40$ & $0$ & $20$ & $0$ & $0$ & $0$ & $4$ & $0$ & $0$
\vstrut \\
\hline
$3$ & $[3]$ & $3$ & $2$ & $-2$ & $40$ & $40$ & $0$ & $20$ & $0$ & $0$ & $0$ & $4$ & $0$ & $0$
\vstrut \\
\hline
$6$ & $[3,2]$ & $1$ & $0$ & $0$ & $120$ & $0$ & $0$ & $0$ & $0$ & $0$ & $0$ & $0$ & $0$ & $0$
\vstrut \\
\hline
$6$ & $[3,2]$ & $1$ & $0$ & $0$ & $120$ & $0$ & $0$ & $0$ & $0$ & $0$ & $0$ & $0$ & $0$ & $0$
\vstrut \\
\hline
$6$ & $[3,3]$ & $0$ & $-1$ & $-1$ & $40$ & $40$ & $20$ & $0$ & $20$ & $8$ & $8$ & $4$ & $0$ & $2$
\vstrut \\
\hline
$3$ & $[3,3]$ & $0$ & $-1$ & $1$ & $40$ & $40$ & $20$ & $0$ & $20$ & $8$ & $8$ & $4$ & $0$ & $2$
\vstrut \\
\hline
$8$ & $[4]$ & $2$ & $-1$ & $0$ & $180$ & $0$ & $60$ & $0$ & $0$ & $12$ & $0$ & $0$ & $20$ & $0$
\vstrut \\
\hline
$8$ & $[4,2]$ & $0$ & $-1$ & $0$ & $180$ & $180$ & $0$ & $0$ & $0$ & $0$ & $0$ & $36$ & $0$ & $0$
\vstrut \\
\hline
$10$ & $[5]$ & $1$ & $0$ & $1$ & $144$ & $144$ & $24$ & $24$ & $24$ & $0$ & $0$ & $0$ & $4$ & $0$
\vstrut \\
\hline
$5$ & $[5]$ & $1$ & $0$ & $-1$ & $144$ & $144$ & $24$ & $24$ & $24$ & $0$ & $0$ & $0$ & $4$ & $0$
\vstrut \\
\hline
$12$ & $[6]$ & $0$ & $1$ & $\sqrt{-3}$ & $120$ & $0$ & $20$ & $0$ & $0$ & $0$ & $0$ & $0$ & $0$ & $2$
\vstrut \\
\hline
$12$ & $[6]$ & $0$ & $1$ & $-\sqrt{-3}$ & $120$ & $0$ & $20$ & $0$ & $0$ & $0$ & $0$ & $0$ & $0$ & $2$
\vstrut \\
\hline
\end{tabular}
\end{minipage}
}} \rotatebox{90}{\box1}
\end{table}

It is immediate to see from Table~\ref{table:characters} that $\mathbb{U}_4$ is an irreducible
representation of the group~\mbox{$2.\mathfrak{S}_6$}.
Using the information provided by Table~\ref{table:characters}, we immediately
obtain the following results.

\begin{corollary}\label{corollary:characters}
Let $\Gamma$ be a subgroup of $2.\mathfrak{S}_6$. After restriction to
the subgroup~$\Gamma$ the \mbox{$2.\mathfrak{S}_6$-representation}~$\mathbb{U}_4$
\begin{itemize}
\item[(i)] remains irreducible,
if $\Gamma$ is one of the subgroups $2.\mathfrak{A}_6$, $2.\mathfrak{S}_5^{nst}$,
$2.\mathfrak{A}_5^{nst}$, $2.\mathfrak{S}_4^{nst}$, $2.F_{36}$, or $2.F_{20}$;
\item[(ii)] splits into a sum of two non-isomorphic irreducible
two-dimensional representations,
if $\Gamma$ is one of the subgroups $2.\mathfrak{A}_5^{st}$, $2.\mathfrak{A}_4^{nst}$,
or $2.\mathrm{D}_{12}^{nst}$.
\end{itemize}
\end{corollary}
\begin{proof}
Compute inner products of the corresponding characters with themselves,
and keep in mind that neither of the groups $2.\mathfrak{A}_5^{st}$, $2.\mathfrak{A}_4^{nst}$,
and~$2.\mathrm{D}_{12}^{nst}$ has an irreducible three-dimensional
representation with a non-trivial action of the central subgroup.
\end{proof}

\begin{remark}
By Corollary~\ref{corollary:characters}(i), the $2.\mathfrak{S}_5^{nst}$-representation
$\mathbb{U}_4$ is irreducible. One can check that it is
not induced from any proper subgroup of $2.\mathfrak{S}_5^{nst}$,
i.e. it defines a primitive subgroup isomorphic to $2.\mathfrak{S}_5$ in
$\mathrm{GL}_4(\mathbb{C})$. Note that this subgroup is not present in the list
given in~\cite[\S8.5]{Feit}. It is still listed by some other classical
surveys, see e.g.~\cite[\S119]{Bli17}.
\end{remark}

\begin{corollary}\label{corollary:characters-P4}
Let $\Gamma$ be a subgroup of $\mathfrak{S}_6$. After restriction to
the subgroup~$\Gamma$ the \mbox{$\mathfrak{S}_6$-representation}~$\mathbb{W}_5$
\begin{itemize}
\item[(i)] remains irreducible,
if $\Gamma$ is one of the subgroups $\mathfrak{A}_6$, $\mathfrak{S}_5^{nst}$,
or $\mathfrak{A}_5^{nst}$;

\item[(ii)] splits into a sum of the trivial and an irreducible
four-dimensional representation if $\Gamma$ is a subgroup~$\mathfrak{A}_5^{st}$;

\item[(iii)] splits into a sum of the trivial and two different
irreducible two-dimensional representations if $\Gamma$ is
a subgroup $\mathrm{D}_{12}^{nst}$.
\end{itemize}
\end{corollary}

In the sequel we will denote the restrictions of the
$2.\mathfrak{S}_6$-representation
$\mathbb{U}_4$ and of the $\SS_6$-representation $\mathbb{W}_5$
to various subgroups by the same symbols for simplicity.
The next two corollaries are implied by direct computations
(we used GAP software \cite{GAP} to perform them).

\begin{corollary}\label{corollary:Sym-U4}
The following assertions hold:
\begin{itemize}
\item[(i)]
the $\mathfrak{A}_6$-representation $\mathrm{Sym}^2(\mathbb{U}_4^\vee)$
does not contain one-dimensional subrepresentations;

\item[(ii)]
the $\mathfrak{A}_6$-representation $\mathrm{Sym}^4(\mathbb{U}_4^\vee)$
does not contain one-dimensional subrepresentations;

\item[(iii)] the $\mathfrak{A}_5^{nst}$-representation $\mathrm{Sym}^2(\mathbb{U}_4^\vee)$
splits into a sum of two different irreducible three-dimensional
representations
and one irreducible four-dimensional representation;

\item[(iv)] the $2.\mathfrak{A}_5^{nst}$-representation $\mathrm{Sym}^3(\mathbb{U}_4^\vee)$
does not contain one-dimensional subrepresentations;

\item[(v)] the $\mathfrak{A}_5^{nst}$-representation $\mathrm{Sym}^4(\mathbb{U}_4^\vee)$
has a unique two-dimensional subrepresentation, and this subrepresentation
splits into a sum of two trivial representations of~$\mathfrak{A}_5^{nst}$.
\end{itemize}
\end{corollary}

Recall that all representations of a symmetric group are self-dual.
Therefore, to study invariant hypersurfaces in $\mathbb{P}(\mathbb{W}_5)$
we will use the following result.

\begin{corollary}\label{corollary:Sym-W5}
Let $\Gamma$ be one of the groups
$\mathfrak{S}_6$, $\mathfrak{A}_6$ or $\mathfrak{S}_5^{nst}$. Then
\begin{itemize}
\item[(i)]
the $\Gamma$-representation $\mathrm{Sym}^2(\mathbb{W}_5)$
has a unique one-dimensional subrepresentation;
\item[(ii)]
the $\Gamma$-representation $\mathrm{Sym}^4(\mathbb{W}_5)$
has a unique two-dimensional subrepresentation, and this subrepresentation
splits into a sum of two trivial
representations of~$\Gamma$.
\end{itemize}
\end{corollary}

We conclude this section by recalling some information about
several subgroups of~\mbox{$2.\mathfrak{S}_6$} that are smaller than those
listed in Table~\ref{table:characters}. Namely, we list in
Table~\ref{table:A5-nst} orders, types and numbers of elements
in certain subgroups of $2.\mathfrak{A}_5^{nst}$. We keep the notation
used in Table~\ref{table:characters}.
By $2.\mathfrak{S}_3^\prime$ we denote a subgroup of~\mbox{$2.\mathfrak{A}_5^{nst}$}
isomorphic to~\mbox{$2.\mathfrak{S}_3$}.
Note that the preimage in $2.\mathfrak{S}_6$ of
any subgroup~\mbox{$\mumu_5\subset\mathfrak{S}_6$} is isomorphic to~$\mumu_{10}$.

\begin{table}[h]
\centering
\caption{Subgroups of $2.\mathfrak{A}_5^{nst}$\label{table:A5-nst}}
\begin{tabular}{|c|c||c|c|c|c|}
\hline
$\mathrm{ord}$ & type &
$2.\mathrm{D}_{10}$ &
$2.\mathfrak{S}_3^\prime$ &
$2.(\mumu_2\times\mumu_2)$ & $\mumu_{10}$ \vstrut\\
\hline
$1$ & $\mathrm{id}$ & $1$ & $1$ & $1$ & $1$ \vstrut\\
\hline
$2$ & $z$ & $1$ & $1$ & $1$ & $1$ \vstrut\\
\hline
$4$ & $[2,2]$ & $10$ & $6$ & $6$ & $0$ \vstrut\\
\hline
$6$ & $[3,3]$ & $0$ & $2$ & $0$ & $0$ \vstrut\\
\hline
$3$ & $[3,3]$ & $0$ & $2$ & $0$ & $0$ \vstrut\\
\hline
$10$ & $[5]$ & $4$ & $0$ & $0$ & $4$ \vstrut\\
\hline
$5$ & $[5]$ & $4$ & $0$ & $0$ & $4$ \vstrut\\
\hline
\end{tabular}
\end{table}

Looking at Table~\ref{table:A5-nst} (and keeping in mind
character values provided by Table~\ref{table:characters})
we immediately obtain the following.

\begin{corollary}\label{corollary:A5-reps}
Let $\Gamma$ be a subgroup of $2.\mathfrak{A}_5^{nst}\subset 2.\mathfrak{S}_6$.
After restriction to
$\Gamma$ the $2.\mathfrak{S}_6$-representation~$\mathbb{U}_4$
\begin{itemize}
\item[(i)] splits into a sum of two non-isomorphic
irreducible two-dimensional representations
if $\Gamma$ is a subgroup $2.\mathrm{D}_{10}$;

\item[(ii)] splits into a sum of an
irreducible two-dimensional representation
and two non-isomorphic one-dimensional representations
if $\Gamma$ is a subgroup $2.\mathfrak{S}_3^\prime$;

\item[(iii)] splits into a sum of two isomorphic
irreducible two-dimensional representations
if $\Gamma$ is a subgroup $2.(\mumu_2\times\mumu_2)$;

\item[(iv)] splits into a sum of four pairwise non-isomorphic
one-dimensional representations
if $\Gamma$ is a subgroup $\mumu_{10}$.
\end{itemize}
\end{corollary}

\section{Icosahedral group in three dimensions}
\label{section:icosahedron}

In this section, we consider the action of the group
$\mathfrak{A}_5$ on the projective space $\mathbb{P}^3$
arising from a non-standard embedding of $\mathfrak{A}_5\hookrightarrow\mathfrak{S}_6$.
Namely, we identify $\mathbb{P}^3$ with the projectivization~\mbox{$\mathbb{P}(\mathbb{U}_4)$},
where $\mathbb{U}_4$ is the restriction of the four-dimensional irreducible
representation of the group~\mbox{$2.\mathfrak{S}_6$}
introduced in Section~\ref{section:representations} to a
subgroup~$2.\mathfrak{A}_5^{nst}$ (which we will refer to just as
$2.\mathfrak{A}_5$ in this section).
Recall from Corollary~\ref{corollary:characters}(i) that $\mathbb{U}_4$ is
an irreducible representation
of~\mbox{$2.\mathfrak{A}_5$}.
Moreover, this is the unique faithful
four-dimensional irreducible representation of the group~$2.\A_5$
(see e.\,g.~\cite[p.~2]{Atlas}).

\begin{remark}[{see e.\,g.~\cite[p.~2]{Atlas}}]
\label{remark:A5-subgroups}
Let $\Gamma$ be a proper subgroup of $\mathfrak{A}_5$ such that the index of $\Gamma$ is at most $15$.
Then $\Gamma$ is isomorphic either to $\mathfrak{A}_4$, or to $\mathrm{D}_{10}$,
or to $\mathfrak{S}_3$, or to $\mumu_5$, or to $\mumu_2\times\mumu_2$.
In particular, if $\mathfrak{A}_5$ acts transitively on the set of $r<15$ elements,
then~\mbox{$r\in\{5,6,10,12\}$}.
\end{remark}

\begin{lemma}
\label{lemma:A5-small-orbits}
Let $\Omega$ be an $\mathfrak{A}_5$-orbit of length $r\leqslant 15$ in $\mathbb{P}^3$.
Then either $r=10$, or $r=12$.
Moreover, $\mathbb{P}^3$ contains exactly two $\mathfrak{A}_5$-orbits of length
$10$ and exactly two $\mathfrak{A}_5$-orbits of length $12$.
\end{lemma}

\begin{proof}
By Remark~\ref{remark:A5-subgroups}
one has $r\in\{1,5,6,10,12,15\}$.
The case $r=1$ is impossible since~$\mathbb{U}_4$
is an irreducible $2.\mathfrak{A}_5$-representation.
Restricting $\mathbb{U}_4$ to subgroups of $2.\mathfrak{A}_5$ isomorphic to
$2.\mathfrak{A}_4$, $2.\mathrm{D}_{10}$,
and~\mbox{$2.(\mumu_2\times\mumu_2)$}, and applying
Corollaries~\ref{corollary:characters}(ii)
and~\ref{corollary:A5-reps}(i),(iii),
we see that~\mbox{$r\not\in\{5,6,15\}$}.

Restricting $\mathbb{U}_4$ to a subgroup of $2.\mathfrak{A}_5$ isomorphic
to $2.\mathfrak{S}_3$, applying Corollary~\ref{corollary:A5-reps}(ii)
and keeping in mind that
there are ten subgroups isomorphic to $\mathfrak{S}_3$ in $\mathfrak{A}_5$,
we see that $\mathbb{P}^3$ contains exactly two $\mathfrak{A}_5$-orbits of length $10$.

Finally, restricting $\mathbb{U}_4$ to a subgroup
of $2.\mathfrak{A}_5$ isomorphic to $\mumu_{10}$,
applying Corollary~\ref{corollary:A5-reps}(iv) and keeping in mind that
there are six subgroups isomorphic to $\mumu_{5}$ in $\mathfrak{A}_5$,
we see that $\mathbb{P}^3$ contains exactly two $\mathfrak{A}_5$-orbits of length $12$.
\end{proof}

\begin{lemma}
\label{lemma:P3-S5-invariants}
There are no $\mathfrak{A}_5$-invariant surfaces of degree at most three in $\mathbb{P}^3$.
\end{lemma}

\begin{proof}
Apply Corollary~\ref{corollary:Sym-U4}(iii),(iv).
\end{proof}

By Corollary~\ref{corollary:characters}(ii), the subgroup
$\mathfrak{A}_4\subset\mathfrak{A}_5$ leaves invariant two disjoint lines
in $\mathbb{P}^3$, say $L_1$ and $L_1^\prime$.
Let $L_1,\ldots,L_5$ be the $\mathfrak{A}_5$-orbit of the line $L_1$,
and let $L_1^\prime,\ldots,L_5^\prime$
be the $\mathfrak{A}_5$-orbit of the line $L_1^\prime$.

\begin{lemma}\label{lemma:5-disjoint}
The lines $L_1,\ldots,L_5$ (respectively,
the lines $L_1^\prime,\ldots,L_5^\prime$) are pairwise disjoint.
\end{lemma}
\begin{proof}
Suppose that some of the lines $L_1,\ldots,L_5$ have a common point.
Since the action of $\mathfrak{A}_5$ on the set $\{L_1,\ldots,L_5\}$ is doubly
transitive, this implies that every two of the lines $L_1,\ldots,L_5$
have a common point.
Therefore, either all lines $L_1,\ldots,L_5$ are coplanar,
or all of them pass through one point. Both of these cases are
impossible since the
$2.\mathfrak{A}_5$-representation $\mathbb{U}_4$ is irreducible
by Corollary~\ref{corollary:characters}(i).
Therefore, the lines $L_1,\ldots,L_5$ are pairwise disjoint.
The same argument applies to the lines $L_1^\prime,\ldots,L_5^\prime$.
\end{proof}

\begin{corollary}\label{corollary:five-lines-orbits}
Any $\mathfrak{A}_5$-orbit contained in the union $L_1\cup\ldots\cup L_5$
has length at least~$20$.
\end{corollary}
\begin{proof}
Corollary~\ref{corollary:characters}(ii) implies that the stabilizer
$\Gamma\cong\mathfrak{A}_4$ of the line $L_1$ acts on $L_1$ faithfully. Therefore,
the length of any $\Gamma$-orbit contained in $L_1$
is at least four. Thus the required assertion follows
from Lemma~\ref{lemma:5-disjoint}.
\end{proof}

We are going to describe the configuration formed
by the lines $L_1,\ldots,L_5,L_1^\prime,\ldots,L_5^\prime$.

\begin{definition}\label{definition:double-five}
Let $T_1,\ldots,T_5,T_1^\prime,\ldots,T_5^\prime$ be different
lines in a projective space. We say that they form a \emph{double
five configuration} if the following conditions hold:
\begin{itemize}
\item the lines $T_1,\ldots,T_5$ (respectively, the lines
$T_1^\prime,\ldots,T_5^\prime$) are pairwise disjoint;

\item for every $i$ the lines $T_i$ and $T_i^\prime$ are disjoint;

\item for every $i\neq j$ the line $T_i$ meets the line $T_j^\prime$.
\end{itemize}
\end{definition}

\begin{lemma}
\label{lemma:double-five}
The lines $L_1,\ldots,L_5,L_1^\prime,\ldots,L_5^\prime$ form a double five configuration. Moreover, the only line in $\mathbb{P}^3$ that intersects all lines
of $L_1,\ldots,L_5$ but $L_i$ is the line $L_i^\prime$,
and the only line in $\mathbb{P}^3$ that intersects all lines
of $L_1^\prime,\ldots,L_5^\prime$ but $L_i^\prime$ is the line $L_i$.
\end{lemma}

\begin{proof}
For any $i$ the lines $L_i$ and $L_i^\prime$ are disjoint by construction.
The lines $L_1,\ldots,L_5$ (respectively,
the lines $L_1^\prime,\ldots,L_5^\prime$) are pairwise disjoint
by Lemma~\ref{lemma:5-disjoint}.

Since any three pairwise skew lines in $\mathbb{P}^3$ are contained in a smooth
quadric surface, and an intersection of two different quadric surfaces
in $\mathbb{P}^3$ cannot contain three pairwise skew lines, we see that for
any three indices $1\leqslant i<j<k\leqslant 5$ there is a unique quadric
surface $Q_{ijk}$ in $\mathbb{P}^3$ passing through the lines $L_i$, $L_j$ and $L_k$.
Moreover, the quadric $Q_{ijk}$ is smooth.
Note also that the quadric $Q_{ijk}$ is not $\mathfrak{A}_5$-invariant by
Lemma~\ref{lemma:P3-S5-invariants}. This implies that all five lines
$L_1,\ldots, L_5$ are not contained in a quadric.

Therefore, we may assume that the quadric $Q_{123}$ does not contain
the line $L_4$. It is well-known that in this case either there is a unique
line $L$ meeting all four lines~\mbox{$L_1,\ldots,L_4$}, or
there are exactly two lines $L$ and $L^\prime$ meeting $L_1,\ldots,L_4$.
In the latter case the stabilizer~\mbox{$\Gamma\subset\mathfrak{A}_5$} of the
quadruple $L_1,\ldots,L_4$ (i.e. the stabilizer of the line $L_5$)
preserves the lines $L_5$, $L$ and $L^\prime$. On the other hand,
the lines $L$ and $L^\prime$ are different from $L_5$
since~$L_5$ meets neither
of the lines $L_1,\ldots,L_4$; moreover, the group $\Gamma\cong\mathfrak{A}_4$
fixes both~$L$ and~$L^\prime$. But $\Gamma$ cannot fix three
different lines in $\mathbb{P}^3$ by Corollary~\ref{corollary:characters}(ii).
The contradiction shows that there is a unique line
$L$ meeting $L_1,\ldots,L_4$. Again we see that $L\neq L_5$, so that
$L=L_5^\prime$ by Corollary~\ref{corollary:characters}(ii).

Since the group $\mathfrak{A}_5$ permutes the lines $L_1,\ldots,L_5$ transitively,
we conclude that the only line in $\mathbb{P}^3$ that intersects all lines
of $L_1,\ldots,L_5$ except $L_i$ is the line $L_i^\prime$.
Similarly, we see that the only line in $\mathbb{P}^3$ that intersects all lines
of $L_1^\prime,\ldots,L_5^\prime$ except $L_i^\prime$ is the line $L_i$.
In particular, the lines $L_1,\ldots,L_5,L_1^\prime,\ldots,L_5^\prime$
form a double five configuration.
\end{proof}

\begin{lemma}
\label{lemma:two-twisted-cubics}
Every $\mathfrak{A}_5$-invariant curve of degree at most three in $\mathbb{P}^3$ is a twisted cubic.
Moreover, there are exactly two $\mathfrak{A}_5$-invariant twisted cubic curves
in $\mathbb{P}^3$.
\end{lemma}

\begin{proof}
Let $C$ be an $\mathfrak{A}_5$-invariant curve of degree at most three in $\mathbb{P}^3$.
Since the \mbox{$2.\mathfrak{A}_5$-representation}~$\mathbb{U}_4$ is irreducible, we conclude that $C$ is a twisted cubic.

By Corollary~\ref{corollary:Sym-U4}(iii), one has
\begin{equation}\label{eq:Sym2-splitting}
\mathrm{Sym}^2(\mathbb{U}_4)\cong\mathbb{V}_3\oplus\mathbb{V}_3^\prime\oplus\mathbb{V}_4,
\end{equation}
where $\mathbb{V}_3$, $\mathbb{V}_3^\prime$, and $\mathbb{V}_4$,
are irreducible representations of the group $\mathfrak{A}_5$
of dimensions~$3$, $3$, and~$4$, respectively.
Note that $\mathbb{V}_3$ and $\mathbb{V}_3^\prime$ are not isomorphic.

Denote by $\mathcal{Q}$ and $\mathcal{Q}^\prime$ the linear systems of quadrics in $\mathbb{P}^3$ that correspond
to $\mathbb{V}_3$ and~$\mathbb{V}_3^\prime$, respectively.
Since $\mathbb{P}^3$ does not contain $\mathfrak{A}_5$-orbits of lengths
less than or equal to eight by Lemma~\ref{lemma:A5-small-orbits},
we see that the base loci of $\mathcal{Q}$ and $\mathcal{Q}^\prime$ contain $\mathfrak{A}_5$-invariant curves
$\mathscr{C}^1$ and $\mathscr{C}^2$, respectively.
The degrees of these curves must be less than four, so that they are twisted cubic curves.
This also implies that the base loci of $\mathcal{Q}$ and $\mathcal{Q}^\prime$ are exactly the curves $\mathscr{C}^1$ and $\mathscr{C}^2$, respectively.

Now take an arbitrary $\mathfrak{A}_5$-invariant twisted cubic curve $C$ in $\mathbb{P}^3$.
The quadrics in $\mathbb{P}^3$ passing through $C$ define a three-dimensional $\mathfrak{A}_5$-subrepresentation in $\mathrm{Sym}^2(\mathbb{U}_4)$.
Moreover, different $\mathfrak{A}_5$-invariant twisted cubics give different $\mathfrak{A}_5$-subrepresentations of $\mathrm{Sym}^2(\mathbb{U}_4)$.
Thus,~\eqref{eq:Sym2-splitting} implies that $C$ coincides either with $\mathscr{C}^1$
or with $\mathscr{C}^2$.
\end{proof}

Keeping in mind Lemma~\ref{lemma:two-twisted-cubics}, we will denote the two
$\mathfrak{A}_5$-invariant twisted cubic curves in $\mathbb{P}^3$ by $\mathscr{C}^1$ and $\mathscr{C}^2$
throughout this section.

\begin{remark}
\label{remark:cubics-disjoint}
The curves $\mathscr{C}^1$ and $\mathscr{C}^2$ are disjoint.
Indeed, otherwise, their intersection would contain at least $12$ points,
which is impossible, since a twisted cubic curve is an intersection of quadrics.
\end{remark}

The lines in $\mathbb{P}^3$ that are tangent to the curves $\mathscr{C}^1$ and $\mathscr{C}^2$
sweep out quartic surfaces $\mathcal{S}^1$ and $\mathcal{S}^2$, respectively.
These surfaces are $\mathfrak{A}_5$-invariant.
The singular loci of $\mathcal{S}^1$ and $\mathcal{S}^2$ are the curves $\mathscr{C}^1$ and $\mathscr{C}^2$, respectively.
In particular, the surfaces $\mathcal{S}^1$ and $\mathcal{S}^2$ are different.
Their singularities along these curves are locally isomorphic to a product of $\mathbb{A}^1$ and an ordinary cusp.

Denote by $\mathcal{P}$ the pencil of quartics in $\mathbb{P}^3$
generated by $\mathcal{S}^1$ and $\mathcal{S}^2$.

\begin{lemma}
\label{lemma:P3-pencil-1}
All $\mathfrak{A}_5$-invariant quartic surfaces in $\mathbb{P}^3$ are contained in
the pencil $\mathcal{P}$.
\end{lemma}
\begin{proof}
This follows from Corollary~\ref{corollary:Sym-U4}(v).
\end{proof}

We proceed by describing the base locus of the pencil $\mathcal{P}$.
This was done in~\cite[Remark~2.6]{CheltsovPrzyjalkowskiShramov}, but we reproduce the details here for the convenience of the reader.

\begin{lemma}
\label{lemma:P3-pencil-base-locus}
The base locus of the pencil $\mathcal{P}$ is an irreducible curve $B$ of degree $16$.
It has $24$ singular points, these points are in a union of two $\mathfrak{A}_5$-orbits of length~$12$,
and each of them is an ordinary cusp of the curve $B$.
The curve $B$ contains a unique $\mathfrak{A}_5$-orbit of length $20$.
\end{lemma}

\begin{proof}
Denote by $B$ the base curve of the pencil $\mathcal{P}$.
Let us show that the curves $\mathscr{C}^1$ and~$\mathscr{C}^2$
are not contained in $B$.
Since $\mathscr{C}^1$ is projectively normal, there is an exact sequence of $2.\mathfrak{A}_5$-representations
$$
0\to H^0(\mathcal{I}_{\mathscr{C}^1}\otimes\mathcal{O}_{\mathbb{P}^3}(4))\to
H^0(\mathcal{O}_{\mathbb{P}^3}(4))\to H^0(\mathcal{O}_{\mathscr{C}^1}\otimes\mathcal{O}_{\mathbb{P}^3}(4))\to 0,%
$$
where $\mathcal{I}_{\mathscr{C}^1}$ is the ideal sheaf of $\mathscr{C}^1$.
The $2.\mathfrak{A}_5$-representation $H^0(\mathcal{O}_{\mathscr{C}^1}\otimes\mathcal{O}_{\mathbb{P}^3}(4))$
contains a one-dimensional subrepresentation corresponding to the unique $\mathfrak{A}_5$-orbit of length $12$ in~\mbox{$\mathscr{C}^1\cong\mathbb{P}^1$}.
This shows that $\mathcal{P}$ contains a surface that does not pass through $\mathscr{C}^1$,
so that~$\mathscr{C}^1$ is not contained in $B$.
Similarly, we see that $\mathscr{C}^2$ is not contained in $B$.

Let $\rho\colon\hat{\mathcal{S}}^1\to\mathcal{S}^1$ be the normalization of the surface $\mathcal{S}^1$,
and let $\hat{\mathscr{C}}^1$ be the preimage of the curve $\mathscr{C}^1$ via $\rho$.
Then the action of the group~$\mathfrak{A}_5$
lifts to~$\hat{\mathcal{S}}^1$.
One has~\mbox{$\hat{\mathcal{S}}^1\cong\mathbb{P}^1\times\mathbb{P}^1$},
and~\mbox{$\rho^*(\mathcal{O}_{\mathcal{S}^1}\otimes\mathcal{O}_{\mathbb{P}^3}(1))$} is a divisor of bi-degree~\mbox{$(1,2)$}.
This shows that $\hat{\mathscr{C}}^1$ is of bi-degree~\mbox{$(1,1)$}.
Thus, the action of $\mathfrak{A}_5$ on $\hat{\mathcal{S}}$ is diagonal
by~\mbox{\cite[Lemma~6.4.3(i)]{CheltsovShramov}}.

Denote by $\hat{B}$ be the preimage of the curve $B$ via $\rho$.
Then $\hat{B}$ is a divisor of bi-degree~\mbox{$(4,8)$}.
Hence, the curve $\hat{B}$ is irreducible by \cite[Lemma~6.4.4(i)]{CheltsovShramov},
so that the curve $B$ is irreducible as well.

Note that the curve $\hat{B}$ is singular.
Indeed, the intersection $\mathcal{S}^1\cap\mathscr{C}^2$
is an $\mathfrak{A}_5$-orbit~$\Sigma_{12}$ of length $12$,
because $\mathscr{C}^2$ is not contained in $\mathcal{S}^1$.
Similarly, we see that the intersection~\mbox{$\mathcal{S}^2\cap\mathscr{C}^1$}
is also an $\mathfrak{A}_5$-orbit $\Sigma_{12}^\prime$ of length $12$.
These $\mathfrak{A}_5$-orbits $\Sigma_{12}$ and $\Sigma_{12}^\prime$ are different by Remark~\ref{remark:cubics-disjoint}.
Since $B$ is the scheme theoretic intersection of the
surfaces~$\mathcal{S}^1$ and~$\mathcal{S}^2$,
it must be singular at every point of $\Sigma_{12}\cup\Sigma_{12}^\prime$.
Denote by $\hat{\Sigma}_{12}$ and $\hat{\Sigma}_{12}^\prime$  the preimages via $\rho$
of the $\mathfrak{A}_5$-orbits $\Sigma_{12}$ and $\Sigma_{12}^\prime$, respectively.
Then $\hat{B}$ is singular in every point of $\hat{\Sigma}_{12}^\prime$.

The curve $\hat{B}$ is smooth away of $\hat{\Sigma}_{12}^\prime$,
because its arithmetic genus is $21$, and the surface~$\hat{\mathcal{S}}^1$
does not contain $\mathfrak{A}_5$-orbits of length less than $12$.
On the other hand, we have
$$
\hat{B}\cap\hat{\mathscr{C}}^1=\hat{\Sigma}_{12},
$$
because $\hat{B}\cdot\mathscr{C}^1=12$ and $\hat{\Sigma}_{12}\subset\hat{B}$.
This shows that $B$ is an irreducible curve whose only singularities are the points of $\Sigma_{12}\cup\Sigma_{12}^\prime$,
and each such point is an ordinary cusp of the curve~$B$.
In particular, the genus of the normalization of the curve $B$ is $9$.
By~\mbox{\cite[Lemma~5.1.5]{CheltsovShramov}}, this implies that $B$ contains a unique $\mathfrak{A}_5$-orbit of length $20$.
\end{proof}

The following classification of $\mathfrak{A}_5$-invariant quartic surfaces in $\mathbb{P}^3$ was obtained
in~\mbox{\cite[Theorem~2.4]{CheltsovPrzyjalkowskiShramov}}.

\begin{lemma}
\label{lemma:P3-pencil-2}
The pencil $\mathcal{P}$ contains two surfaces
$\mathcal{R}^1$ and $\mathcal{R}^2$ with ordinary double
singularities, such that the singular loci
of $\mathcal{R}^1$ and $\mathcal{R}^2$ are $\mathfrak{A}_5$-orbits of length $10$.
Every surface in $\mathbb{P}^3$ different from
$\mathcal{S}^1$, $\mathcal{S}^2$, $\mathcal{R}^1$ and $\mathcal{R}^2$
is smooth.
\end{lemma}

\begin{proof}
Let $S$ be a surface in $\mathcal{P}$ that is different from $\mathcal{S}^1$ and $\mathcal{S}^2$.
It follows from Lemma~\ref{lemma:P3-S5-invariants} that $S$ is irreducible.
Assume that $S$ is singular.

We claim that $S$ has isolated singularities.
Indeed, suppose that $S$ is singular along some $\mathfrak{A}_5$-invariant curve $Z$.
Taking a general plane section of $S$, we see that the degree of~$Z$ is at most three.
Thus, one has either $Z=\mathscr{C}^1$ or $Z=\mathscr{C}^2$ by Lemma~\ref{lemma:two-twisted-cubics}.
Since neither of these curves is contained in the base locus of
$\mathcal{P}$ by Lemma~\ref{lemma:P3-pencil-base-locus},
this would imply that either $S=\mathcal{S}^1$ or $S=\mathcal{S}^2$.
The latter is not the case by assumption.

We see that the singularities of $S$ are isolated.
Hence, $S$ contains at most two non-Du Val singular points by \cite[Theorem~1]{Umezu}
applied to the minimal resolution of singularities of the surface $S$.
Since the set of all non-Du Val singular points of the surface $S$ must be $\mathfrak{A}_5$-invariant,
we see that $S$ has none of them by Lemma~\ref{lemma:A5-small-orbits}.
Thus, all singularities of $S$ are Du Val.

By \cite[Lemma~6.7.3(iii)]{CheltsovShramov}, the surface $S$
has only ordinary double singularities,
the set~\mbox{$\mathrm{Sing}(S)$} consists of one $\mathfrak{A}_5$-orbit, and
$$
\big|\mathrm{Sing}(S)\big|\in\big\{5,6,10,12,15\big\}.
$$
Since $\mathbb{P}^3$ does not contain $\mathfrak{A}_5$-orbits of lengths
$5$, $6$, and $15$ by Lemma~\ref{lemma:A5-small-orbits},
we see that~\mbox{$\mathrm{Sing}(S)$} is either
an $\mathfrak{A}_5$-orbit of length $10$ or an $\mathfrak{A}_5$-orbit of length $12$.

Suppose that the singular locus of $S$ is an $\mathfrak{A}_5$-orbit
$\Sigma_{12}$ of length $12$.
Then $S$ does not contain other $\mathfrak{A}_5$-orbits of length $12$ by \cite[Lemma~6.7.3(iv)]{CheltsovShramov}.
Since $\mathscr{C}^1$ is not contained in the base locus of $\mathcal{P}$
by Lemma~\ref{lemma:P3-pencil-base-locus},
and $\mathscr{C}^1$ is contained in $\mathcal{S}^1$, we see that~\mbox{$\mathscr{C}^1\not\subset S$}.
Since
$$S\cdot\mathscr{C}^1=12$$
and $\Sigma_{12}$ is the only $\mathfrak{A}_5$-orbit of length at most $12$ in~\mbox{$\mathscr{C}^1\cong\mathbb{P}^1$},
we have~\mbox{$S\cap\mathscr{C}^1=\Sigma_{12}$}.
Thus,
$$
12=S\cdot\mathscr{C}^1\geqslant \sum_{P\in\Sigma_{12}}\mathrm{mult}_{P}(S)=2|\Sigma_{12}|=24,
$$
which is absurd.

Therefore, we see that the singular locus of $S$
is an $\mathfrak{A}_5$-orbit $\Sigma_{12}$ of length $10$.
Vice versa, if an $\mathfrak{A}_5$-invariant quartic surface
passes through an $\mathfrak{A}_5$-orbit of length $10$, then
it is singular by~\cite[Lemma~6.7.1(ii)]{CheltsovShramov}.
We know from Lemma~\ref{lemma:A5-small-orbits} that there are exactly
two $\mathfrak{A}_5$-orbits of length $10$ in $\mathbb{P}^3$,
and it follows from Lemma~\ref{lemma:P3-pencil-base-locus} that they are not contained
in the base locus of $\mathcal{P}$. Thus there are
two surfaces $\mathcal{R}^1$ and $\mathcal{R}^2$
that are singular exactly at the points of these two
$\mathfrak{A}_5$-orbits, respectively. The above argument
shows that every surface in $\mathcal{P}$ except
$\mathcal{S}^1$, $\mathcal{S}^2$, $\mathcal{R}^1$ and $\mathcal{R}^2$
is smooth.
\end{proof}

Keeping in mind Lemma~\ref{lemma:P3-pencil-2},
we will denote by $\mathcal{R}^1$ and $\mathcal{R}^2$
the two nodal surfaces contained in the pencil~$\mathcal{P}$
until the end of this section.

\begin{lemma}
\label{lemma:no-quartic-through-ten-lines}
There is a unique $\mathfrak{A}_5$-invariant quartic surface in $\mathbb{P}^3$ that contains the lines $L_1,\ldots,L_5$ (respectively, the lines $L_1^\prime,\ldots,L_5^\prime$).
Moreover, this surface is smooth, and it does not contain
the lines $L_1^\prime,\ldots,L_5^\prime$ (respectively, $L_1,\ldots,L_5$).
\end{lemma}

\begin{proof}
Put $\mathcal{L}=\sum_{i=1}^5L_i$ and
$\mathcal{L}^\prime=\sum_{i=1}^5L_i^\prime$.
Corollary~\ref{corollary:characters}(ii) implies that
the stabilizer in $\mathfrak{A}_5$ of a general point of $L_1$ is trivial.
Therefore, there exists a surface $S\in\mathcal{P}$
that contains all lines $L_1,\ldots,L_5$.
By Lemma~\ref{lemma:P3-pencil-base-locus} such surface $S$ is unique.

We claim that $S\neq\mathcal{S}^1$. Indeed, all lines
contained in $\mathcal{S}^1$ are tangent to the curve
$\mathscr{C}^1$, and there are no $\mathfrak{A}_5$-orbits of length five
in $\mathscr{C}^1\cong\mathbb{P}^1$. Similarly, one has
$S\neq\mathcal{S}^2$.

We claim that $S$ is not one of the two nodal surfaces
$\mathcal{R}^1$ and $\mathcal{R}^2$
contained in the pencil~$\mathcal{P}$.
Indeed, suppose that $S=\mathcal{R}^1$.
Since the singular locus
of $\mathcal{R}^1$ is an $\mathfrak{A}_5$-orbit of length $10$
by Lemma~\ref{lemma:P3-pencil-2}, we see that
the lines $L_1,\ldots,L_5$ are contained in the smooth locus of $\mathcal{R}^1$ by Corollary~\ref{corollary:five-lines-orbits}.
On the other hand, one has $\mathcal{L}^2=-10$
by Lemma~\ref{lemma:5-disjoint}.
This means that~\mbox{$\mathrm{rk}\,\mathrm{Pic}(S)^{\mathfrak{A}_5}\geqslant 2$},
which is impossible by \cite[Lemma~6.7.3(i),(ii)]{CheltsovShramov}.

We see that the surface $S$ is different from $\mathcal{R}^1$.
The same argument shows that $S$ is different from $\mathcal{R}^2$.
Hence, $S$ is smooth
by Lemma~\ref{lemma:P3-pencil-2}.

Let us show that $S$ does not contain the
lines $L_1^\prime,\ldots,L_5^\prime$.
Suppose that it does.
By Lemma~\ref{lemma:5-disjoint}
one has
$$\mathcal{L}\cdot\mathcal{L}=\mathcal{L}^\prime\cdot\mathcal{L}^\prime
=-10.$$
By \cite[Lemma~6.7.1(i)]{CheltsovShramov},
we have $\mathrm{rk}\,\mathrm{Pic}(S)^{\mathfrak{A}_5}=2$.
Let $\Pi_S$ be the class of a plane section of $S$.
Then the determinant of the matrix
$$
\left(\begin{matrix} %
\mathcal{L}\cdot \mathcal{L}        &  \mathcal{L}\cdot \mathcal{L}^\prime      & \Pi_S\cdot \mathcal{L}\cr%
\mathcal{L}\cdot \mathcal{L}^\prime &\mathcal{L}^\prime\cdot \mathcal{L}^\prime & \Pi_S\cdot \mathcal{L}^\prime\cr%
\Pi_S\cdot \mathcal{L}& \Pi_S\cdot \mathcal{L}^\prime& \Pi_S\cdot \Pi_S\cr%
\end{matrix}\right)=
\left(\begin{matrix} %
-10&  20 & 5\cr%
 20& -10& 5\cr%
5& 5& 4\cr%
\end{matrix}\right)
$$
must vanish. This is a contradiction, because it equals $300$.

Applying the same argument, we see that the lines
$L_1^\prime,\ldots,L_5^\prime$ are contained in a unique
\mbox{$\mathfrak{A}_5$-invariant} quartic surface, this surface is smooth and
does not contain the lines~\mbox{$L_1,\ldots,L_5$}.
\end{proof}

\begin{remark}
\label{remark:double-five}
One can use the properties of the pencil $\mathcal{P}$ to give an alternative proof of Lemma~\ref{lemma:double-five}.
Namely, we know from Lemma~\ref{lemma:no-quartic-through-ten-lines}
that there are two (different) smooth $\mathfrak{A}_5$-invariant
quartic surfaces $S$ and $S^\prime$ containing the lines
$L_1,\ldots,L_5$ and $L_1^\prime,\ldots,L_5^\prime$, respectively.
By Lemma~\ref{lemma:P3-pencil-base-locus},
the base locus of the pencil $\mathcal{P}$ is an irreducible curve $B$
that contains a unique $\mathfrak{A}_5$-orbit $\Sigma$ of length $20$.
By Corollary~\ref{corollary:five-lines-orbits}, this implies that $\Sigma$ is contained in
the union~\mbox{$L_1\cup\ldots\cup L_5$}, because
$$
B\cdot (L_1+\ldots+L_5)=20
$$
on the surface $S$. Similarly, we see that $\Sigma$ is contained in $L_1^\prime\cup\ldots\cup L_5^\prime$.
These facts together with Lemma~\ref{lemma:5-disjoint} easily imply
that the lines $L_1,\ldots,L_5$ and $L_1^\prime,\ldots,L_5^\prime$ form a double five configuration.
\end{remark}

Now we will obtain some restrictions on low degree $\mathfrak{A}_5$-invariant curves in $\mathbb{P}^3$.

\begin{lemma}
\label{lemma:A5-invariant-curves-small-degree}
Let $C$ be an irreducible $\mathfrak{A}_5$-invariant curve in $\mathbb{P}^3$ of degree $d\leqslant 10$.
Denote by $g$ the genus of the normalization of the curve $C$. Then
$$
g\leqslant\frac{d^2}{8}+1-|\mathrm{Sing}(C)|.
$$
\end{lemma}

\begin{proof}
Since $\mathbb{U}_4$ is an irreducible $2.\mathfrak{A}_5$-representation, the curve $C$ is not contained in a plane in $\mathbb{P}^3$.
This implies that a stabilizer in $\mathfrak{A}_5$ of a general point of the curve $C$ is trivial.
In particular, the $\mathfrak{A}_5$-orbit of a general point of $C$ has length $|\mathfrak{A}_5|=60$.

Let $S$ be a surface in the pencil $\mathcal{P}$ that passes through a general point of $C$.
Then the curve $C$ is contained in $S$, because otherwise one would have
$$
60\leqslant |S\cap C|\leqslant S\cdot C=4d\leqslant 40,
$$
which is absurd.
Since the assertion of the lemma clearly holds for the twisted cubic curves $\mathscr{C}^1$ and $\mathscr{C}^2$,
we may assume that $C$ is different from these two curves.

Suppose that $S=\mathcal{S}^1$.
Let us use the notation of the proof of Lemma~\ref{lemma:P3-pencil-base-locus}.
Denote by~$\hat{C}$ the preimage of the curve $C$ via $\rho$.
Then $\hat{C}$ is a divisor of bi-degree $(a,b)$ for some non-negative integers $a$ and $b$ such that $d=2a+b$.
On the other hand, one has
$$
|\hat{C}\cap\hat{\mathscr{C}}^1|\leqslant \hat{C}\cdot\hat{\mathscr{C}}^1=a+b\leqslant 2a+b=d\leqslant 10,
$$
which is impossible, since the curve $\hat{\mathscr{C}}^1\cong\mathscr{C}^1\cong\mathbb{P}^1$ does not contain $\mathfrak{A}_5$-orbits of length less than $12$.

We see that $S\ne\mathcal{S}^1$.
Similarly, we see that $S\ne\mathcal{S}^2$.
By Lemma~\ref{lemma:P3-pencil-2}, either $S$ is a smooth quartic $K3$ surface, or $S$ is one of the surfaces $\mathcal{R}^1$ and $\mathcal{R}^2$.
Denote by $\Pi_S$ a plane section of~$S$. Then
$$
\mathrm{det}\left(
  \begin{array}{cc}
    \Pi_S^2 & \Pi_S\cdot C \\
    \Pi_S\cdot C & C^2 \\
  \end{array}
\right)=\mathrm{det}\left(
  \begin{array}{cc}
    4 & d \\
    d & C^2 \\
  \end{array}
\right)=4C^2-d^2\leqslant 0
$$
by the Hodge index theorem.

Suppose that $C$ is contained in the smooth locus of the surface $S$.
Denote by $p_a(C)$ the arithmetic genus of the curve $C$. Then
$$
C^2=2p_a(C)-2.
$$
by the adjunction formula. Thus, we get
$$
p_a(C)\leqslant\frac{d^2}{8}+1.
$$
Since $g\leqslant p_a(C)-|\mathrm{Sing}(C)|$, this implies the assertion of the lemma.

To complete the proof, we may assume that $C$ contains a singular
point of the surface~$S$.
By Lemma~\ref{lemma:P3-pencil-2}, this means that either $S=\mathcal{R}^1$ or $S=\mathcal{R}^2$.
The singularities of the surface $S$ are ordinary double points,
and its singular locus is an $\mathfrak{A}_5$-orbit of length~$10$.
In particular, the curve $C$ contains the whole singular locus of~$S$.
By \cite[Lemma~6.7.3(i),(ii)]{CheltsovShramov}, one has $\mathrm{Pic}(S)^{\mathfrak{A}_5}\cong\mathbb{Z}$.
Since $\Pi_S^2=4$ and the self-intersection of any Cartier divisor
on the surface $S$ is even,
we see that the group $\mathrm{Pic}(S)^{\mathfrak{A}_5}$ is generated by $\Pi_S$.

Suppose that $C$ is a Cartier divisor on $S$.
Then either $C\sim\Pi_S$ or $C\sim 2\Pi_S$, because~\mbox{$d\leqslant 10$}.
Since the restriction map
\begin{equation*}
H^0\big(\mathcal{O}_{\mathbb{P}^3}(n)\big)\to H^0\big(\mathcal{O}_S(n\Pi_S)\big)
\end{equation*}
is an isomorphism for $n\leqslant 3$, we conclude that
there is an $\mathfrak{A}_5$-invariant quadric in $\mathbb{P}^3$. This is not the case
by Lemma~\ref{lemma:P3-S5-invariants}.

Therefore, we see that $C$ is not a Cartier divisor on $S$.
Since $S$ has only ordinary double points, the divisor $2C$ is Cartier.
Thus
$$
2C\sim l\Pi_S,
$$
for some odd positive integer $l$. Since
$$
2d=2C\cdot\Pi_S=l\Pi_S\cdot\Pi_S=4l,
$$
we see that $l=\frac{d}{2}$. In particular, $d$ is even and $l\leqslant 5$.

Let $\theta\colon\tilde{S}\to S$ be the minimal resolution of singularities of the surface~$S$.
Denote by $\tilde{C}$ the proper transform of the curve~$C$ on the surface $\tilde{S}$,
and denote by $\Theta_1,\ldots,\Theta_{10}$ the exceptional curves of $\theta$. Then
$$
2\tilde{C}\sim \theta^*(l\Pi_S)-m\sum_{i=1}^{10}\Theta_i,
$$
for some positive integer $m$. Moreover, $m$ is odd, because $C$ is not a Cartier divisor. We have
$$
4\tilde{C}^2=\Pi_{S}^2l^2-20m^2=4l^2-20m^2,
$$
which implies that $\tilde{C}^2=l^2-5m^2$.
Since $\tilde{C}^2$ is even, $m$ is odd and  $l\leqslant 5$,
we see that either~\mbox{$l=3$} or $l=5$.

Denote by $p_a(\tilde{C})$ the arithmetic genus of the curve $\tilde{C}$. Then
$$
l^2-5m^2=\tilde{C}^2=2p_a(\tilde{C})-2.
$$
by the adjunction formula.
In particular, we have
$$25-5m^2\geqslant l^2-5m^2\geqslant -2,$$
so that $l\in\{3,5\}$ and $m=1$. The latter means that
$C$ is smooth at every point of $\mathrm{Sing}(S)$,
so that
$$|\mathrm{Sing}(\tilde{C})|=|\mathrm{Sing}(C)|.$$
If $l=\frac{d}{2}=3$, then $p_a(\tilde{C})=3$. This gives
$$
g\leqslant p_a(\tilde{C})-|\mathrm{Sing}(\tilde{C})|=3-|\mathrm{Sing}(C)|\leqslant\frac{d^2}{8}+1-|\mathrm{Sing}(C)|.
$$
Similarly, if $l=\frac{d}{2}=5$, then $p_a(\tilde{C})=11$. This gives
$$
g\leqslant p_a(\tilde{C})-|\mathrm{Sing}(\tilde{C})|=11-|\mathrm{Sing}(C)|\leqslant\frac{d^2}{8}+1-|\mathrm{Sing}(C)|.
$$
\end{proof}

Recall from \cite[Lemma~5.4.1]{CheltsovShramov} that there exists a unique smooth irreducible curve of genus $4$ with a faithful action of the group~$\mathfrak{A}_5$.
This curve is known as the Bring's curve.
Its canonical model is a complete intersection of a quadric and a cubic in a three-dimensional projective space.
However, this sextic curve does not appear in our $\mathbb{P}^3=\mathbb{P}(\mathbb{U}_4)$ by

\begin{lemma}\label{lemma:S5-on-g-4}
Let $C$ be a smooth irreducible $\mathfrak{A}_5$-invariant curve in $\mathbb{P}^3$ of degree $d\leqslant 6$ and genus $g$. Then~\mbox{$g\neq 4$}.
\end{lemma}

\begin{proof}
Suppose that $g=4$.
Denote by $\Pi_C$ the plane section of the curve $C$.
Then
$$
h^0(\mathcal{O}_C(\Pi_C))=d-3+h^0(\mathcal{O}_C(K_C-\Pi_C))
$$
by the Riemann--Roch theorem.
Since $C$ is not contained in a plane, this implies
that~\mbox{$\Pi_C\sim K_C$}.
Therefore, the projective space~$\mathbb{P}^3$ is identified with a projectivization
of an \mbox{$\mathfrak{A}_5$-representation}~\mbox{$H^0(\mathcal{O}_C(K_C))^\vee$},
i.e. of a representation of the group~$2.\mathfrak{A}_5$ where the center
of~$2.\mathfrak{A}_5$ acts trivially. The latter is not the case by
construction of~$\mathbb{U}_4$.
\end{proof}

\begin{lemma}
\label{lemma:A5-deg-10}
Let $C$ be an irreducible smooth $\mathfrak{A}_5$-invariant curve in $\mathbb{P}^3$ of degree $d=10$ and genus $g$.
Then $g\ne 10$.
\end{lemma}

\begin{proof}
Suppose that $g=10$.
By Lemma~\ref{lemma:P3-pencil-base-locus},
the base locus of the pencil $\mathcal{P}$ is an irreducible curve $B$
of degree $16$.
In particular, there exists a surface $S\in\mathcal{P}$ that does not contain $C$.
Thus, the intersection~\mbox{$S\cap C$}
is an $\mathfrak{A}_5$-invariant set that consists of
$$
C\cdot S=4d=40
$$
points (counted with multiplicities).
On the other hand, by \cite[Lemma~5.1.5]{CheltsovShramov}, any $\mathfrak{A}_5$-orbit in $C$ has length $12$, $30$, or $60$.
\end{proof}

\section{Large subgroups of $\mathfrak{S}_6$}
\label{section:basic-actions}

In this section we collect some auxiliary results about the groups
$\mathfrak{S}_6$, $\mathfrak{A}_6$ and $\mathfrak{S}_5$.
We start with recalling some general properties of the group $\mathfrak{A}_6$.

\begin{remark}[{see e.\,g.~\cite[p.~4]{Atlas}}]
\label{remark:A6-subgroups}
Let $\Gamma$ be a proper subgroup of $\mathfrak{A}_6$ such that the
index of~$\Gamma$ is at most $15$. Then
$\Gamma$ is isomorphic either to $\mathfrak{A}_5$, or to $F_{36}$, or to $\mathfrak{S}_4$.
In particular, if~$\mathfrak{A}_6$ acts transitively on the set of $r<15$
elements, then either $r=6$ or $r=10$.
\end{remark}

We will need the following result about possible actions
of the group $\mathfrak{A}_6$ on curves of small genera (cf. \cite[Theorem~2.18]{ChSh09b} and~\cite[Lemma~5.1.5]{CheltsovShramov}).

\begin{lemma}
\label{lemma:sporadic-genera} Suppose that $C$ is a smooth
irreducible curve of genus $g\leqslant 15$ with a non-trivial
action of the group $\mathfrak{A}_6$.
Then $g=10$.
\end{lemma}

\begin{proof}
Let $\Omega\subset C$ be an $\mathfrak{A}_6$-orbit. Then a stabilizer
of a point in $\Omega$ is a cyclic subgroup of $\mathfrak{A}_6$, which implies that
$$
|\Omega|\in\big\{72, 90, 120, 180, 360\big\}.
$$

From the classification of finite subgroups
of $\mathrm{Aut}(\mathbb{P}^1)\cong\mathrm{PGL}_2(\mathbb{C})$ we know that $g\neq 0$.
Also, it follows from the non-solvability of the~group $\mathfrak{A}_6$ that
$g\neq 1$.

Put $\bar{C}=C\slash \mathfrak{A}_6$. Then $\bar{C}$ is a smooth curve.
Let~$\bar{g}$ be the genus of the curve~$\bar{C}$. The
Riemann--Hurwitz formula gives
$$
2g-2=360\big(2\bar{g}-2\big)+180a_{180}+240a_{120}+270a_{90}+288a_{72},
$$
where $a_k$ is the~number of $\mathfrak{A}_6$-orbits in $C$
of length~$k$.

Since $a_k\geqslant 0$ and $2\leqslant g\leqslant 15$, one has $\bar{g}=0$.
Thus, we obtain
$$
2g-2=-720+180a_{180}+240a_{120}+270a_{90}+288a_{72}.
$$
Going through the values $2\leqslant g\leqslant 15$, and solving this equation
case by case we see that the only possibility is
$g=10$.
\end{proof}

We proceed by recalling some general properties of the group $\mathfrak{S}_5$.

\begin{remark}[{see e.\,g.~\cite[p.~2]{Atlas}}]
\label{remark:S5-subgroups}
Let $\Gamma$ be a proper subgroup of $\mathfrak{S}_5$ such that the
index of $\Gamma$ is less than $12$. Then
$\Gamma$ is isomorphic either to $\mathfrak{A}_5$, or to $\mathfrak{S}_4$,
or to $F_{20}$, or to $\mathfrak{A}_4$, or to $\mathrm{D}_{12}$.
In particular, if $\mathfrak{S}_5$ acts transitively on the set of $r<12$
elements, then $r\in\{2,5,6,10\}$.
\end{remark}

\begin{lemma}
\label{lemma:S5-on-g-5}
The group $\mathfrak{S}_5$ cannot act faithfully on a smooth irreducible curve of genus~$5$.
\end{lemma}

\begin{proof}
Suppose that $C$ is a curve of genus $5$
with a faithful action of $\mathfrak{S}_5$. Considering the action
of the subgroup $\mathfrak{A}_5\subset\mathfrak{S}_5$ on $C$ and applying
\cite[Lemma~5.4.3]{CheltsovShramov}, we see that $C$ is hyperelliptic.
This gives a natural homomorphism
$$
\theta\colon \mathfrak{S}_5\to\mathrm{Aut}(\mathbb{P}^1)\cong\mathrm{PGL}_2(\mathbb{C})
$$
whose kernel is either trivial or isomorphic to $\mumu_2$. Thus $\theta$ is injective, which gives a contradiction.
\end{proof}

Now we will prove some auxiliary facts about actions of the
groups $\mathfrak{S}_6$, $\mathfrak{A}_6$ and $\mathfrak{S}_5$ on the four-dimensional
projective space.

\begin{remark}\label{remark:pohuj}
The group $\mathfrak{S}_6$ has exactly four irreducible five-dimensional
representations (see e.\,g.~\cite[p.~5]{Atlas}).
Starting from one of them, one more can be obtained by a twist by an outer automorphism of~$\mathfrak{S}_6$,
and two remaining ones are obtained from these two by a tensor product with the sign representation.
Although these four representations
are not isomorphic, the images of $\mathfrak{S}_6$ in $\mathrm{PGL}_5(\mathbb{C})$ under them are the same.
Every irreducible five-dimensional
representation of $\mathfrak{S}_6$ restricts to an irreducible representation of
the subgroup $\mathfrak{A}_6\subset\mathfrak{S}_6$, and restricts to an irreducible representation of
the \emph{some} of the subgroups $\mathfrak{S}_5\subset\mathfrak{S}_6$.
The group $\mathfrak{A}_6$ has exactly two irreducible five-dimensional representations,
each of them arising this way (see e.\,g.~\cite[p.~5]{Atlas}).
Similarly, the group $\mathfrak{S}_5$ has exactly two irreducible five-dimensional
representations, each of them arising this way (see e.\,g.~\cite[p.~2]{Atlas}).
Note also that every five-dimensional representation
of a group $\mathfrak{A}_6$ or $\mathfrak{S}_5$ that does not contain one-dimensional
subrepresentations is irreducible.
\end{remark}

Let $\mathbb{V}_5$ be an irreducible five-dimensional
representation of the group~$\mathfrak{S}_6$.
Put~\mbox{$\mathbb{P}^4=\mathbb{P}(\mathbb{V}_5)$}.
Keeping in mind Remark~\ref{remark:pohuj},
we see that the image of the corresponding homomorphism
$\mathfrak{S}_6$ to $\mathrm{PGL}_5(\mathbb{C})$ is the same for any choice of $\mathbb{V}_5$,
and thus the $\mathfrak{S}_6$-orbits and $\mathfrak{S}_6$-invariant hypersurfaces in
$\mathbb{P}^4$ do not depend on $\mathbb{V}_5$ either.

Remark~\ref{remark:pohuj} implies that there are six
linear forms $x_0,\ldots,x_5$ on $\mathbb{P}^4$
that are permuted by the group $\mathfrak{S}_6$
(cf. Sections~\ref{section:intro} and~\ref{section:representations}).
Indeed, up to a twist by an outer automorphoism
of $\mathfrak{S}_6$ and a tensor product with the sign representation,
$\mathbb{V}_5$ is a subrepresentation of the six-dimensional
representation $\mathbb{W}$ of~$\mathfrak{S}_6$, so that one can take
restricitions of the natural coordinates in $\mathbb{W}$
to be these linear forms.
Let $Q$ be the three-dimensional quadric
in $\mathbb{P}^4$ given by equation
\begin{equation}\label{eq:Q}
x_0^2+x_1^2+x_2^2+x_3^2+x_4^2+x_5^2=0.
\end{equation}
The quadric $Q$ is smooth and $\mathfrak{S}_6$-invariant.
Note also that equation~\eqref{equation:quartic} makes sense in our~$\mathbb{P}^4$.

We will use the notation introduced above until the end of the paper.

\begin{lemma}
\label{lemma:A6-S5-quartics}
Let $\Gamma$ be either the group
$\mathfrak{S}_6$, or its subgroup $\mathfrak{A}_6$, or a subgroup $\mathfrak{S}_5$ of $\mathfrak{S}_6$
such that $\mathbb{V}_5$ is an irreducible representation of $\Gamma$. Then
the only $\Gamma$-invariant quadric threefold in
$\mathbb{P}^4$ is the quadric $Q$.
Similarly, every (reduced) $\Gamma$-invariant quartic threefold
in $\mathbb{P}^4$ is given
by equation~\eqref{equation:quartic} for some $t\in\mathbb{C}$.
\end{lemma}

\begin{proof}
Apply Corollary~\ref{corollary:Sym-W5}.
\end{proof}

By a small abuse of notation we will refer to the
points in $\mathbb{P}^4$ using $x_i$ as if they were homogeneous coordinates,
i.e. a point in $\mathbb{P}^4$ will be encoded by a ratio
of six linear forms~$x_i$.
As in Section~\ref{section:intro}, let $\Sigma_{6}$ and $\Sigma_{10}$ be
the $\mathfrak{S}_6$-orbits of the points~\mbox{$[-5:1:1:1:1:1]$}
and~\mbox{$[-1:-1:-1:1:1:1]$},
respectively. Looking at equation~\eqref{eq:Q}, we obtain

\begin{corollary}\label{corollary:Q-orbits}
The quadric $Q$ does not contain the $\mathfrak{S}_6$-orbits $\Sigma_6$ and $\Sigma_{10}$.
\end{corollary}

Now we will have a look at the action of the group $\mathfrak{A}_6$ on $\mathbb{P}^4$.
Note that $\mathbb{V}_5$ is an irreducible $\mathfrak{A}_6$-representation
by Remark~\ref{remark:pohuj}.

\begin{lemma}
\label{lemma:A6-orbits-in-P4}
There are no $\mathfrak{A}_6$-orbits of length less than six in $\mathbb{P}^4$.
Moreover, the only $\mathfrak{A}_6$-orbit of length six in $\mathbb{P}^4$ is $\Sigma_6$.
\end{lemma}

\begin{proof}
The only subgroup of $\mathfrak{A}_6$ of index less than six is $\mathfrak{A}_6$ itself
(cf. Remark~\ref{remark:A6-subgroups}),
so that the first assertion of the lemma follows from irreducibility of
the $\mathfrak{A}_6$-representation~$\mathbb{V}_5$.
Also, the only subgroups of $\mathfrak{A}_6$ of index six are $\mathfrak{A}_5^{st}$ and $\mathfrak{A}_5^{nst}$,
so that the second assertion of the lemma also follows from Corollary~\ref{corollary:characters-P4}.
\end{proof}

\begin{lemma}
\label{lemma:X-7-10}
Let $X$ be an $\mathfrak{A}_6$-invariant quartic threefold
in $\mathbb{P}^4$ that contains an $\mathfrak{A}_6$-orbit of length at most six.
Then~\mbox{$X=X_{\frac{7}{10}}$}.
\end{lemma}

\begin{proof}
By Lemma~\ref{lemma:A6-S5-quartics}, one has $X=X_t$
for some $t\in\mathbb{C}$, and by Lemma~\ref{lemma:A6-orbits-in-P4}
the $\mathfrak{A}_6$-orbit $\Sigma_6$ is contained in $X_t$.
Since $\Sigma_6$ is not contained in the quadric $Q$ by
Corollary~\ref{corollary:Q-orbits}, we see that there is a unique $t\in\mathbb{C}$
such that $\Sigma_6$ is contained in a quartic given by
equation~\eqref{equation:quartic}. Therefore, we conclude that
$t=\frac{7}{10}$.
\end{proof}

Now we will make a couple of observations about the action of the group
$\mathfrak{S}_5$ on $\mathbb{P}^4$. We choose $\mathfrak{S}_5$ to be a subgroup of $\mathfrak{S}_6$
such that $\mathbb{V}_5$ is an irreducible $\mathfrak{S}_5$-representation
(cf.~Remark~\ref{remark:pohuj} and Corollary~\ref{corollary:characters-P4}).

\begin{lemma}
\label{lemma:P4-S5-orbits}
Let $P\in\mathbb{P}^4$ be a point such that its stabilizer in $\mathfrak{S}_5$
contains a subgroup isomorphic to $\mathrm{D}_{12}$. Then the
$\mathfrak{S}_5$-orbit of $P$ is $\Sigma_{10}$.
\end{lemma}

\begin{proof}
By Corollary~\ref{corollary:characters-P4}(iii),
the point in $\mathbb{P}^4$ fixed by a subgroup $\mathrm{D}_{12}\subset\mathfrak{S}_5$
is unique. On the other hand, it is straightforward to check
that a stabilizer in $\mathfrak{S}_5$ of a point of $\Sigma_{10}$ contains
a subgroup isomorphic to $\mathrm{D}_{12}$. It remains to notice that
the latter stabilizer is actually isomorphic to $\mathrm{D}_{12}$,
since the only subgroups of $\mathfrak{S}_5$ that contain $\mathrm{D}_{12}$
are $\mathrm{D}_{12}$ and $\mathfrak{S}_5$ itself, while $\mathfrak{S}_5$ has no fixed
points on~$\mathbb{P}^4$.
\end{proof}

\begin{lemma}
\label{lemma:X-1-6}
Let $X$ be an $\mathfrak{S}_5$-invariant quartic threefold in $\mathbb{P}^4$ that contains $\Sigma_{10}$.
Then~\mbox{$X=X_{\frac{1}{6}}$}.
\end{lemma}

\begin{proof}
By Lemma~\ref{lemma:A6-S5-quartics}, one has $X=X_t$
for some $t\in\mathbb{C}$.
Since $\Sigma_{10}$ is not contained in the quadric $Q$ by Corollary~\ref{corollary:Q-orbits},
we see that there is a unique $t\in\mathbb{C}$ such that $\Sigma_{10}$ is contained in a quartic given by equation~\eqref{equation:quartic}.
Therefore, we conclude that $t=\frac{1}{6}$.
\end{proof}

\section{Rationality of the quartic threefold $X_{\frac{7}{10}}$}
\label{section:7-10}

In this section we will construct an explicit
$\mathfrak{A}_6$-equivariant birational map $\mathbb{P}^3\dasharrow X_{\frac{7}{10}}$.
Implicitly, the construction of this map first appeared in the proof of \cite[Theorem~1.20]{ChSh09b}.
Here we will present a much simplified proof of its existence.

We identify $\mathbb{P}^3$ with the projectivization~\mbox{$\mathbb{P}(\mathbb{U}_4)$},
where $\mathbb{U}_4$ is the restriction of the four-dimensional irreducible
representation of the group~\mbox{$2.\mathfrak{S}_6$}
introduced in Section~\ref{section:representations} to the
subgroup~$2.\mathfrak{A}_6$. By Corollary~\ref{corollary:characters}(i),
the $2.\mathfrak{A}_6$-representation $\mathbb{U}_4$ is irreducible.

\begin{lemma}
\label{lemma:A6-invariants-in-P3}
There are no $\mathfrak{A}_6$-invariant surfaces of odd degree in $\mathbb{P}^3$,
and no $\mathfrak{A}_6$-invariant pencils of surfaces of odd degree in
$\mathbb{P}^3$. Moreover, there are no $\mathfrak{A}_6$-invariant quadric and quartic surfaces
in $\mathbb{P}^3$.
\end{lemma}

\begin{proof}
Recall that the only one-dimensional representation of the group
$2.\mathfrak{A}_6$ is the trivial representation.
Therefore, any $\mathfrak{A}_6$-invariant surface of odd degree $d$ in $\mathbb{P}^3$
gives rise to a trivial $2.\mathfrak{A}_6$-subrepresentation in
$R_d=\mathrm{Sym}^d(\mathbb{U}_4)$. On the other hand, the non-trivial central
element $z$ of $2.\mathfrak{A}_6$ acts on $R_d$ by a scalar matrix with diagonal
entries equal to~\mbox{$-1$}, which shows that $R_d$ does not contain
trivial $2.\mathfrak{A}_6$-representations.
Also, since the only two-dimensional representation of $2.\mathfrak{A}_6$
is the sum of two trivial representations, this implies that
there are no $\mathfrak{A}_6$-invariant pencils of surfaces of odd degree in $\mathbb{P}^3$.

The last assertion of the lemma follows from
Corollary~\ref{corollary:Sym-U4}(i),(ii).
\end{proof}

\begin{lemma}
\label{lemma:P3-A6-orbits}
Let $\Omega$ be an $\mathfrak{A}_6$-orbit in $\mathbb{P}^3$. Then $|\Omega|\geqslant 16$.
\end{lemma}

\begin{proof}
Lemma~\ref{lemma:A6-invariants-in-P3} implies that there are no
$\mathfrak{A}_6$-orbits of odd length in $\mathbb{P}^3$.
Thus, if $\Omega$ is an $\mathfrak{A}_6$-orbit in $\mathbb{P}^3$ of length
at most $15$, then by Remark~\ref{remark:A6-subgroups} a stabilizer
of its general point is isomorphic either to $\mathfrak{A}_5$ or to $F_{36}$.
Both of these cases are impossible by
Corollary~\ref{corollary:characters}.
\end{proof}

Actually, the minimal degree of an $\mathfrak{A}_6$-invariant surface in $\mathbb{P}^3$
equals~$8$ (see~\mbox{\cite[Lemma~3.7]{ChSh09b}}), and
the minimal length of an $\mathfrak{A}_6$-orbit in $\mathbb{P}^3$ equals~$36$
(see~\mbox{\cite[Lemma~3.8]{ChSh09b}}), but we will not need this here.

\begin{lemma}[{cf. \cite[Lemma~4.26]{ChSh09b}}]
\label{lemma:A6-on-g-9}
Let $C$ be a smooth irreducible $\mathfrak{A}_6$-invariant curve
of degree $9$ and genus $g$ in $\mathbb{P}^3$.
Then $g\neq 10$.
\end{lemma}

\begin{proof}
Suppose that $g=10$.
Then it follows from \cite[Example~6.4.3]{Har77} that either $C$ is
contained in a unique quadric surface in $\mathbb{P}^3$,
or $C$ is a complete intersection of two cubic surfaces in $\mathbb{P}^3$.
The former case is impossible, since there are no $\mathfrak{A}_6$-invariant quadrics in~$\mathbb{P}^3$ by Lemma~\ref{lemma:A6-invariants-in-P3}.
The latter case is impossible, because there are no $\mathfrak{A}_6$-invariant
pencils of cubic surfaces in $\mathbb{P}^3$ by Lemma~\ref{lemma:A6-invariants-in-P3}.
\end{proof}

Recall that the group $\mathfrak{A}_6$ contains six standard subgroups isomorphic to $\mathfrak{A}_5$
and six non-standard subgroups isomorphic to $\mathfrak{A}_5$ (see the conventions made in Section~\ref{section:representations}).
Denote the former ones by $H_1^\prime,\ldots,H_6^\prime$,
and denote the latter ones by $H_1,\ldots,H_6$.
By  Corollary~\ref{corollary:characters}(ii), each group $H_i^\prime$ leaves invariant two lines $L_i^1$ and $L_i^2$ in $\mathbb{P}^3$.
Note that each group $H_i$ permutes transitively the lines $L_1^1,\ldots,L_6^1$ (respectively, $L_1^2,\ldots,L_6^2$).

Put $\mathcal{L}^1=L_1^1+\ldots+L_6^1$ and $\mathcal{L}^2=L_1^2+\ldots+L_6^2$.
Then the curves $\mathcal{L}^1$ and $\mathcal{L}^2$ are $\mathfrak{A}_6$-invariant,
and the curve $\mathcal{L}^1+\mathcal{L}^2$ is $\mathfrak{S}_6$-invariant.

\begin{lemma}
\label{lemma:six-lines}
The lines $L_1^1,\ldots,L_6^1$ (respectively, the lines $L_1^2,\ldots,L_6^2$) are pairwise disjoint.
Moreover, the curves $\mathcal{L}^1$ and $\mathcal{L}^2$ are disjoint.
\end{lemma}

\begin{proof}
We use an argument
similar to one in the proof of Lemma~\ref{lemma:5-disjoint}.
Suppose that some of the lines $L_1^1,\ldots,L_6^1$ have a common point.
Since the action of $\mathfrak{A}_6$ on the set~\mbox{$\{L_1^1,\ldots,L_6^1\}$}
is doubly transitive,
this implies that any two of the lines~\mbox{$L_1^1,\ldots,L_6^1$}
have a common point.
Therefore, either all lines $L_1^1,\ldots,L_6^1$ are coplanar, or all of them
pass through one point. Both of these cases are impossible since
$\mathbb{U}_4$ is an irreducible $2.\mathfrak{A}_6$-representation (see
Corollary~\ref{corollary:characters}(i)).
Therefore, the lines $L_1^1,\ldots,L_6^1$ are pairwise disjoint.
The same argument applies to the lines $L_1^2,\ldots,L_6^2$.

Suppose that some of the lines $L_1^1,\ldots,L_6^1$, say, $L_1^1$, intersects
some of the lines~\mbox{$L_1^2,\ldots,L_6^2$}.
Since the lines $L_1^1$ and $L_2^1$ are
disjoint by construction, we may assume that $L_1^1$ intersects~$L_2^2$.
Since the stabilizer $H_1^\prime\subset\mathfrak{A}_6$ of $L_1^1$ acts transitively on the lines
$L_2^2,\ldots,L_6^2$, we conclude that all five lines $L_2^2,\ldots,L_6^2$ intersect $L_1^1$.
Therefore, the line $L_1^1$ contains a subset of at most five points
that is invariant with respect to the group $H_1^\prime\cong\mathfrak{A}_5$, which is a contradiction.
Thus, $\mathcal{L}^1$ and $\mathcal{L}^2$ are disjoint.
\end{proof}

\begin{lemma}
\label{lemma:A6-invariant-curves-small-degree}
Let $C$ be an $\mathfrak{A}_6$-invariant curve in $\mathbb{P}^3$ of degree $d\leqslant 10$.
Then either $C=\mathcal{L}^1$ or $C=\mathcal{L}^2$.
\end{lemma}

\begin{proof}
Suppose first that $C$ is reducible.
We may assume that  $\mathfrak{A}_6$ permutes the irreducible components of $C$ transitively.
Thus, $C$ has either $6$ or $10$ irreducible components by
Remark~\ref{remark:A6-subgroups},
and these irreducible components are lines.
By Remark~\ref{remark:A6-subgroups} and
Corollary~\ref{corollary:characters} the latter case is impossible,
and in the former case
one has either $C=\mathcal{L}^1$ or $C=\mathcal{L}^2$.

Therefore, we assume that the curve $C$ is irreducible.
Let $g$ be the genus of the normalization of the curve $C$.
We have
\begin{equation}
\label{equation:A6-pa}
g\leqslant\frac{d^2}{8}+1-|\mathrm{Sing}(C)|\leqslant 13-|\mathrm{Sing}(C)|
\end{equation}
by Lemma~\ref{lemma:A5-invariant-curves-small-degree}.
This implies that the curve $C$ is smooth, because $\mathbb{P}^3$ does not contain $\mathfrak{A}_6$-orbits of length less than~$16$ by Lemma~\ref{lemma:P3-A6-orbits}.

If $d\leqslant 8$, then \eqref{equation:A6-pa} gives $g\leqslant 9$.
This is impossible by Lemma~\ref{lemma:sporadic-genera}

If~\mbox{$d=9$}, then~\eqref{equation:A6-pa} gives $g\leqslant 11$,
so that $g=10$ by Lemma~\ref{lemma:sporadic-genera}.
This is impossible by Lemma~\ref{lemma:A6-on-g-9}.

Therefore, we see that $d=10$.
Thus, \eqref{equation:A6-pa} gives $g\leqslant 13$, so that $g=10$ by Lemma~\ref{lemma:sporadic-genera}.
The latter is impossible by Lemma~\ref{lemma:A5-deg-10}.
\end{proof}

Denote by $\mathcal{M}$ the linear system on $\mathbb{P}^3$
consisting of all quartic surfaces passing through the~lines $L_{1}^1,\ldots,L_{6}^1$.
Then $\mathcal{M}$ is not empty. In fact, its dimension is at least four by parameter count.
Moreover, the linear system $\mathcal{M}$ does not have base components by Lemma~\ref{lemma:A6-invariants-in-P3}.

\begin{lemma}
\label{lemma:base-locus-six-lines}
The base locus of $\mathcal{M}$ does not contain curves
except the lines~\mbox{$L_{1}^1,\ldots,L_{6}^1$}.
Moreover, a general surface in $\mathcal{M}$ is smooth at a general
point of each of the lines~\mbox{$L_{1}^1,\ldots,L_{6}^1$}.
\end{lemma}

\begin{proof}
Denote by $Z$ the union of the curves that are contained in the base locus of $\mathcal{M}$ and are different from the lines $L_{1}^1,\ldots,L_{6}^1$.
Then $Z$ is a (possibly empty) $\mathfrak{A}_6$-invariant curve.
Denote its degree by $d$. Pick two general surfaces $M_1$ and $M_2$ in $\mathcal{M}$.
Then
$$
M_1\cdot M_2=Z+m\mathcal{L}^1+\Delta,
$$
where $m$ is a positive integer,
and $\Delta$ is an effective one-cycle on $\mathbb{P}^3$
that contains none of the lines $L_{1}^1,\ldots,L_{6}^1$.
Note that $\Delta$ may contain irreducible components of the curve $Z$.
Let $\Pi$ be a plane in $\mathbb{P}^3$.
Then
$$
16=\Pi\cdot M_1\cdot M_2=\Pi\cdot Z+m\Pi\cdot\mathcal{L}^1+\Pi\cdot\Delta=d+6m+\Pi\cdot\Delta\leqslant d+6m,
$$
which implies that $m\leqslant 2$ and $d\leqslant 10$.
By Lemma~\ref{lemma:A6-invariant-curves-small-degree}, we have $d=0$, so that $Z$ is empty.
Since
$$
2\geqslant m\geqslant\mathrm{mult}_{L_{i}^1}\big(M_1\big)\mathrm{mult}_{L_{i}^1}\big(M_2\big),
$$
we see that a general surface in $\mathcal{M}$ is smooth at a general point of $L_i^1$.
\end{proof}

Let $\alpha\colon U\to\mathbb{P}^{3}$ be a blow up along the lines
$L_{1}^1,\ldots,L_{6}^1$.
Then the action of $\mathfrak{A}_6$ lifts to $U$.
Denote by $E_1,\ldots,E_6$ the $\alpha$-exceptional surfaces that are
mapped to~\mbox{$L_{1}^1,\ldots,L_{6}^1$}, respectively.
Denoting by $\Pi$ a plane in $\P^3$, we compute
$$
(-K_U)^3=\left(4\alpha^*\Pi-\sum\limits_{i=1}^6 E_i\right)^3=
64\left(\alpha^*\Pi\right)^3+12\sum\limits_{i=1}^6 \alpha^*\Pi\cdot E_i^2-
\sum\limits_{i=1}^6 E_i^3=64-72+12=4.
$$

\begin{lemma}
\label{lemma:twisted-diagonal-action}
The action of the stabilizer $H_i^\prime\cong\mathfrak{A}_5$ in $\mathfrak{A}_6$
of the line $L_i^1$ on the surface~\mbox{$E_i\cong\mathbb{P}^1\times\mathbb{P}^1$}
is twisted diagonal, i.e., $E_i$ is identified with $\mathbb{P}(\mathbb{U}_2)\times\mathbb{P}(\mathbb{U}_2^\prime)$,
where~$\mathbb{U}_2$ and~$\mathbb{U}_2^\prime$ are different two-dimensional irreducible representations of the group $2.\mathfrak{A}_5$.
\end{lemma}

\begin{proof}
This follows from Corollary~\ref{corollary:characters}(ii).
\end{proof}

Let us denote by $\mathcal{M}_U$ the proper transform of the linear system $\mathcal{M}$ on the threefold $U$.
Then~\mbox{$\mathcal{M}_U\sim -K_U$} by Lemma~\ref{lemma:base-locus-six-lines}.

\begin{lemma}
\label{lemma:base-locus-six-lines-blow-up}
The linear system $\mathcal{M}_U$ is base point free.
\end{lemma}

\begin{proof}
Let us first show that $\mathcal{M}_U$ is free from base curves.
Suppose that the base locus of the linear system $\mathcal{M}_U$ contains some curves.
Then each of these curves is contained in some of the $\alpha$-exceptional surfaces by Lemma~\ref{lemma:base-locus-six-lines}.
Denote by $Z$ the union of all such curves that are contained in $E_1$.
Then $Z$ is an $H_1^\prime$-invariant curve.
For some non-negative integers~$a$ and~$b$, one has
$$
Z\sim as+bl,
$$
where $s$ is a section of the natural projection $E_1\to L_1^1$ such that $s^2=0$ on $E_1$, and $l$ is a fiber of this projection.
On the other hand, we have
$$
\mathcal{M}_U\vert_{E_1}\sim -K_U\vert_{E_1}\sim s+3l.
$$
This gives $a\leqslant 1$ and $b\leqslant 3$.
Since the action of $H_1^\prime$ on the surface $E_1$ is twisted diagonal by Lemma~\ref{lemma:twisted-diagonal-action},
the latter is impossible by \cite[Lemma~6.4.2(i)]{CheltsovShramov} and \cite[Lemma~6.4.11(o)]{CheltsovShramov}.

We see that  $\mathcal{M}_U$ is free from base curves.
Since $\mathcal{M}_U\sim -K_{U}$, the linear system $\mathcal{M}_U$ cannot have more than $-K_U^3=4$ base points.
By Lemma~\ref{lemma:P3-A6-orbits},
this implies that $\mathcal{M}_U$ is base point free.
\end{proof}

\begin{corollary}
\label{corollary:base-locus-six-lines-blow-up}
The base locus of the linear system $\mathcal{M}$ consists of the lines
$L_{1}^1,\ldots,L_{6}^1$.
\end{corollary}

By Lemma~\ref{lemma:base-locus-six-lines-blow-up}, the divisor $-K_U$ is nef.
Since $-K_U^3=4$, it is also big.
Thus, we have $$
h^1\Big(\mathcal{O}_U\big(-K_U\big)\Big)=h^2\Big(\mathcal{O}_U\big(-K_U\big)\Big)=0
$$
by the Kawamata--Viehweg vanishing theorem (see \cite{Kawamata-vanishing}).
Hence, the Riemann--Roch formula gives
\begin{equation}\label{eq:A6-h0}
h^0(\mathcal{O}_U(-K_U))=5.
\end{equation}
In particular, we see that $|-K_U|=\mathcal{M}_U$.

\begin{lemma}
\label{lemma:anticanonical-linear-system-W5}
The $\mathfrak{A}_6$-representation $H^0(\mathcal{O}_U(-K_U))$
is irreducible.
\end{lemma}

\begin{proof}
By Lemma~\ref{lemma:A6-invariants-in-P3},
there are no $\mathfrak{A}_6$-invariant quartic surfaces in $\mathbb{P}^3$.
This implies that $H^0(\mathcal{O}_U(-K_U))$ does not contain one-dimensional subrepresentations.
Hence it is irreducible by Remark~\ref{remark:pohuj}.
\end{proof}

Lemma~\ref{lemma:base-locus-six-lines-blow-up}
together with~\eqref{eq:A6-h0} implies that
there is an~$\mathfrak{A}_6$-equivariant commutative diagram
\begin{equation}
\label{equation:A6-half-link}
\xymatrix{
&U\ar@{->}[dl]_{\alpha}\ar@{->}[dr]^{\beta}&\\
\mathbb{P}^{3}\ar@{-->}[rr]_{\phi}&&\mathbb{P}^4}
\end{equation}
where $\phi$ is the rational map given by $\mathcal{M}$,
and $\beta$ is a morphism given by the anticanonical linear system $|-K_U|$.
By Lemma~\ref{lemma:anticanonical-linear-system-W5}
the projective space $\mathbb{P}^4$ in~\eqref{equation:A6-half-link} is a projectivization
of an irreducible $\mathfrak{A}_6$-representation.

Recall from Lemma~\ref{lemma:two-twisted-cubics} that $\mathbb{P}^3$ contains exactly two $H_1$-invariant
twisted cubic curves~$\mathscr{C}_1^1$ and~$\mathscr{C}_1^2$.

\begin{lemma}
\label{lemma:six-cubics}
The curve $\mathcal{L}^1$ intersects exactly one curve among $\mathscr{C}_1^1$ and $\mathscr{C}_1^2$.
Moreover, each line among $L_1^1,\ldots,L_6^1$ contains two points of this intersection.
Similarly, the curve~$\mathcal{L}^2$ intersects exactly one curve among $\mathscr{C}_1^1$ and $\mathscr{C}_1^2$,
and this curve is different from the one that intersects $\mathcal{L}^1$.
\end{lemma}

\begin{proof}
By Remark~\ref{remark:A5-subgroups}, the stabilizer in $H_1$ of the curve $L_1^1$ is isomorphic to $\mathrm{D}_{10}$,
and thus it has an orbit of length $2$ on $L_1^1$.
Thus, the curve  $\mathcal{L}^1$ contains an $H_1$-orbit $\Sigma_{12}^1$
of length~$12$ by Lemma~\ref{lemma:A5-small-orbits}.
Similarly, the curve  $\mathcal{L}^2$ contains an $H_1$-orbit $\Sigma_{12}^2$ of length $12$.
By Lemma~\ref{lemma:six-lines}, one has $\Sigma_{12}^1\ne \Sigma_{12}^2$.
Moreover, $\Sigma_{12}^1$ and $\Sigma_{12}^2$ are the only $H_1$-orbits in $\mathbb{P}^3$ of length $12$ by Lemma~\ref{lemma:A5-small-orbits}.
Since $\mathscr{C}_1^1$ and $\mathscr{C}_1^2$ are disjoint by Remark~\ref{remark:cubics-disjoint},
and each of them contains an $H_1$-orbit of length $12$, we see that either $\Sigma_{12}^1\subset\mathscr{C}_1^1$
and $\Sigma_{12}^2\subset\mathscr{C}_1^2$, or $\Sigma_{12}^2\subset\mathscr{C}_1^1$ and $\Sigma_{12}^1\subset\mathscr{C}_1^2$.
Since a line cannot have more than two common points with a twisted cubic,
this also implies the last assertion of the lemma.
\end{proof}

Without loss of generality, we may assume that the curve  $\mathcal{L}^1$ intersects $\mathscr{C}_1^1$,
and the curve $\mathcal{L}^2$ intersects $\mathscr{C}_1^2$.
Let $\mathscr{C}_1^1,\ldots,\mathscr{C}_6^1$ be the $\mathfrak{A}_6$-orbit of the curve $\mathscr{C}_1^1$,
and let $\mathscr{C}_1^2,\ldots,\mathscr{C}_6^2$ be the $\mathfrak{A}_6$-orbit of the curve $\mathscr{C}_1^2$.
By Lemma~\ref{lemma:two-twisted-cubics}, the curves $\mathscr{C}_i^1$ and $\mathscr{C}_i^2$ are the only twisted cubic curves in $\mathbb{P}^3$ that are $H_i$-invariant.
By Lemma~\ref{lemma:six-cubics}, we have

\begin{corollary}
\label{corollary:six-twisted-cubic-lines}
Every twisted cubic curve  $\mathscr{C}_i^1$ intersects each line among $L_1^1,\ldots,L_6^1$ by two points.
Similarly, every twisted cubic curve  $\mathscr{C}_i^2$ intersects each line among $L_1^2,\ldots,L_6^2$ by two points.
\end{corollary}

Denote by $\widetilde{\mathscr{C}}_1^1,\ldots,\widetilde{\mathscr{C}}_6^1$ the proper
transforms on $U$ of the curves $\mathscr{C}_1^1,\ldots,\mathscr{C}_6^1$, respectively.

\begin{lemma}
\label{lemma:six-twisted-curves-contracted}
One has $-K_{U}\cdot\widetilde{\mathscr{C}}_1^1=\ldots=
-K_{U}\cdot\widetilde{\mathscr{C}}_6^1=0$.
\end{lemma}

\begin{proof}
This follows from Corollary~\ref{corollary:six-twisted-cubic-lines}.
\end{proof}

We see that each curve $\widetilde{\mathscr{C}}_i^1$ is contracted by $\beta$ to a point.
Since the $\mathfrak{A}_6$-orbit of $\widetilde{\mathscr{C}}_1^1$ consists of six
curves, we also obtain the following.

\begin{corollary}
\label{corollary:X-7-10-Sigma-6}
The image of the morphism $\beta$ contains an $\mathfrak{A}_6$-orbit of length at most six.
\end{corollary}

Since $-K_U^3=4$, the image of $\beta$ is either an $\mathfrak{A}_6$-invariant quartic threefold or an $\mathfrak{A}_6$-invariant quadric threefold.
Using results of \cite{Vazzana}, one can show that the latter case is impossible.
However, this immediately follows from Corollary~\ref{corollary:X-7-10-Sigma-6}.
Indeed, an $\mathfrak{A}_6$-orbit of length at most six cannot be contained in the $\mathfrak{A}_6$-invariant quadric by Corollary~\ref{corollary:Q-orbits} and Lemma~\ref{lemma:A6-orbits-in-P4}.

\begin{corollary}
\label{corollary:X-7-10-Q}
The morphism $\beta$ is birational onto its image,
and its image is a quartic threefold.
\end{corollary}

Now Lemma~\ref{lemma:X-7-10} implies that the image of $\beta$ is the quartic $X_{\frac{7}{10}}$.
This proves

\begin{corollary}
\label{corollary:X-7-10-rational}
The threefold $X_{\frac{7}{10}}$ is rational.
\end{corollary}

Let us conclude this section by recalling two related results proved in \cite[\S4]{ChSh09b}.
The commutative diagram~\eqref{equation:A6-half-link} can be extended to
an $\mathfrak{A}_6$-equivariant commutative diagram
\begin{equation}
\label{equation:A6-link}
\xymatrix{
&U\ar@{-->}[rrrr]^{\rho}\ar@{->}[ddl]_{\alpha}\ar@{->}[dr]^{\gamma}&&&& U\ar@{->}[ddr]^{\alpha}\ar@{->}[dl]_{\gamma}&\\
&&X_{\frac{7}{10}}\ar@{->}[rr]^{\sigma}&&X_{\frac{7}{10}}&&\\
\mathbb{P}^{3}\ar@{-->}[rrrrrr]_{\psi}\ar@{-->}[rru]_{\phi}&&&&&&\mathbb{P}^{3}\ar@{-->}[ull]^{\phi}}
\end{equation}
Here $\sigma$ is an automorphism of the quartic threefold $X_{\frac{7}{10}}$ given by an odd permutation in $\mathfrak{S}_6$
acting on~$\mathbb{P}^4$, cf. Remark~\ref{remark:pohuj}.
The birational map $\rho$ is a composition of Atiyah flops in $36$ curves contracted by $\gamma$,
and the birational map $\psi$ is not regular.

The diagram~\eqref{equation:A6-link} is a so-called $\mathfrak{A}_6$-Sarkisov link.
The subgroup $\mathfrak{A}_6\subset\mathrm{Aut}(\mathbb{P}^3)$ together with $\psi\in\mathrm{Bir}^{\mathfrak{A}_6}(\mathbb{P}^3)$ generates a subgroup isomorphic to $\mathfrak{S}_6$.
Moreover, the subgroup
$$\mathrm{Aut}^{\mathfrak{A}_6}(\mathbb{P}^3)\subset\mathrm{Bir}^{\mathfrak{A}_6}(\mathbb{P}^3)$$
is also isomorphic to $\mathfrak{S}_6$.
By \cite[Theorem~1.24]{ChSh09b}, the whole group $\mathrm{Bir}^{\mathfrak{A}_6}(\mathbb{P}^3)$ is a free product of these two copies of $\mathfrak{S}_6$
amalgamated over the original $\mathfrak{A}_6$.

\section{Rationality of the quartic threefold $X_{\frac{1}{6}}$}
\label{section:1-6}

In this section we will construct an explicit
$\mathfrak{S}_5$-equivariant birational map $\mathbb{P}^3\dasharrow X_{\frac{1}{6}}$.
We identify $\mathbb{P}^3$ with the projectivization~\mbox{$\mathbb{P}(\mathbb{U}_4)$},
where $\mathbb{U}_4$ is the restriction of the four-dimensional irreducible
representation of the group~\mbox{$2.\mathfrak{S}_6$}
introduced in Section~\ref{section:representations} to a
subgroup~$2.\mathfrak{S}_5^{nst}$, and denote the latter
subgroup simply by $2.\mathfrak{S}_5$. By Corollary~\ref{corollary:characters}(i),
the $2.\mathfrak{S}_5$-representation $\mathbb{U}_4$ is irreducible.

\begin{lemma}
\label{lemma:P3-S5-orbits}
Let $\Omega$ be an $\mathfrak{S}_5$-orbit in $\mathbb{P}^3$. Then $|\Omega|\geqslant 12$.
\end{lemma}

\begin{proof}
Apply Remark~\ref{remark:S5-subgroups} together with
Corollary~\ref{corollary:characters}.
\end{proof}

\begin{lemma}
\label{lemma:S5-invariant-curves-small-degree}
Let $C$ be an $\mathfrak{S}_5$-invariant curve in $\mathbb{P}^3$ of degree $d$. Then $d\geqslant 6$.
\end{lemma}

\begin{proof}
Suppose that $d\leqslant 5$.
To start with, assume that $C$ is reducible and denote by $r$ the number of its irreducible components.
We may assume that  $\mathfrak{S}_5$ permutes the irreducible components of $C$ transitively.
Thus, either $r=2$ or $r=5$ by Remark~\ref{remark:S5-subgroups}.
If $r=5$, the irreducible components of $C$ are lines,
so that this case is impossible by Remark~\ref{remark:S5-subgroups} and Corollary~\ref{corollary:characters}(i).
Hence, we have $r=2$, and the stabilizer of each of the two irreducible components $C_1$ and $C_2$ of $C$ is the subgroup $\mathfrak{A}_5\subset\mathfrak{S}_5$.
Moreover, in this case one has
$$
\mathrm{deg}(C_1)=\mathrm{deg}(C_2)\leqslant 2,
$$
which is impossible by Lemma~\ref{lemma:two-twisted-cubics}.

Therefore, we assume that the curve $C$ is irreducible.
Let $g$ be the genus of the normalization of the curve $C$.
Then
$$
g\leqslant\frac{d^2}{8}+1-|\mathrm{Sing}(C)|
$$
by Lemma~\ref{lemma:A5-invariant-curves-small-degree}, so that
$g\leqslant 4-|\mathrm{Sing}(C)|$.
This implies that $C$ is smooth, because there are no $\mathfrak{S}_5$-orbits of length less than $12$ by Lemma~\ref{lemma:P3-S5-orbits}.

Since $\mathfrak{S}_5$ does not act faithfully on $\mathbb{P}^1$, we see that $g\ne 0$.
Thus, either $g=4$ or $g=5$ by \cite[Lemma~5.1.5]{CheltsovShramov}.
The former case is impossible by Lemma~\ref{lemma:S5-on-g-4},
while the latter case is impossible by Lemma~\ref{lemma:S5-on-g-5}.
\end{proof}

Recall from Section~\ref{section:icosahedron} that
the subgroup $\mathfrak{A}_4\subset\mathfrak{A}_5\subset\mathfrak{S}_5$ fixes
two disjoint lines $L_1$ and $L_1^\prime$.
As before, we consider the $\mathfrak{A}_5$-orbit $L_1,\ldots,L_5$
of the line $L_1$ and the $\mathfrak{A}_5$-orbit $L_1^\prime,\ldots,L_5^\prime$
of the line $L_1^\prime$.
By Lemma~\ref{lemma:double-five} the lines
$L_1,\ldots,L_5,L_1^\prime,\ldots,L_5^\prime$
form a double five configuration
(see Definition~\ref{definition:double-five}).
Corollary~\ref{corollary:characters}(i) implies that
the $\mathfrak{S}_5$-orbit of the line $L_1$
is~\mbox{$L_1,\ldots,L_5,L_1^\prime,\ldots,L_5^\prime$}.

\begin{remark}\label{remark:F20-transitive}
Any subgroup $F_{20}\subset\mathfrak{S}_5$ permutes the ten
lines~\mbox{$L_1,\ldots,L_5,L_1^\prime,\ldots,L_5^\prime$}
transitively. Indeed, let $c\in F_{20}$ be an element of order
five. Then $c$ is not contained in a stabilizer of the line $L_1$,
so that the orbit of $L_1$ with respect to the group
$\Gamma\cong\mumu_5$ generated by $c$ is $L_1,\ldots,L_5$.
Similarly, the $\Gamma$-orbit of the line
$L_1^\prime$ is $L_1^\prime,\ldots,L_5^\prime$.
Also, the group $F_{20}$ is not contained in $\mathfrak{A}_5$, so that
the $F_{20}$-orbit of $L_1$ contains some of the lines
$L_1^\prime,\ldots,L_5^\prime$, and thus contains all the ten
lines~\mbox{$L_1,\ldots,L_5,L_1^\prime,\ldots,L_5^\prime$}.
\end{remark}

Let $\mathcal{M}$ be the linear system on $\mathbb{P}^3$
consisting of all quartic surfaces passing through all lines $L_1,\ldots,L_5$ and $L_1^\prime,\ldots,L_5^\prime$.
Then $\mathcal{M}$ is not empty.
In fact, Lemma~\ref{lemma:double-five} and parameter count imply that its dimension is at least four.
Moreover, the linear system $\mathcal{M}$ does not have base components by Lemma~\ref{lemma:P3-S5-invariants}.

\begin{lemma}
\label{lemma:base-locus-ten-lines}
The base locus of $\mathcal{M}$ does not contain curves that are different from the lines $L_1,\ldots,L_5,L_1^\prime,\ldots,L_5^\prime$.
Moreover, a general surface in $\mathcal{M}$ is smooth in a general point of each of these lines.
Furthermore, two general surfaces in $\mathcal{M}$ intersect transversally at a general point of each of these lines.
\end{lemma}

\begin{proof}
Denote by $Z$ the union of all curves that are contained in the base locus of the linear system
$\mathcal{M}$ and are different from the lines $L_1,\ldots,L_5,L_1^\prime,\ldots,L_5^\prime$.
Then $Z$ is a (possibly empty) $\mathfrak{S}_5$-invariant curve.
Denote its degree by $d$. Pick two general surfaces $M_1$ and $M_2$ in $\mathcal{M}$.
Then
$$
M_1\cdot M_2=Z+m\sum_{i=1}^5L_i+m\sum_{i=1}^5L_i^\prime+\Delta,
$$
where $m$ is a positive integer, and $\Delta$ is an effective one-cycle on $\mathbb{P}^3$
that contains none of the lines $L_1,\ldots,L_5$ and $L_1^\prime,\ldots,L_5^\prime$.
Note that $\Delta$ may contain irreducible components of the curve $Z$.
Note also that $\Delta\ne 0$, because $\mathcal{M}$ is not a pencil.

Let $\Pi$ be a plane in $\mathbb{P}^3$.
Then
$$
16=\Pi\cdot Z+m\sum_{i=1}^5\Pi\cdot L_i+m\sum_{i=1}^5\Pi\cdot L_i^\prime+\Pi\cdot\Delta=d+10m+\Pi\cdot\Delta>d+10m,
$$
which implies that $m=1$ and $d\leqslant 5$.
By Lemma~\ref{lemma:S5-invariant-curves-small-degree},
we have $d=0$, so that $Z$ is empty. Since
$$
1\geqslant m\geqslant\mathrm{mult}_{L_{i}}\big(M_1\cdot M_2\big)\geqslant\mathrm{mult}_{L_{i}}\big(M_1\big)\mathrm{mult}_{L_{i}}\big(M_2\big),
$$
we see that a general surface in $\mathcal{M}$ is smooth at a general point of $L_i$,
and two general surfaces in $\mathcal{M}$ intersect transversally at a general point of $L_i$.
Similarly, we see that a general surface in $\mathcal{M}$ is smooth at a general point of $L_i^\prime$,
and two general surfaces in $\mathcal{M}$ intersect transversally at a general point of $L_i^\prime$.
\end{proof}

Let $g\colon W\to\mathbb{P}^{3}$ be a~blow up along the lines $L_{1},\ldots,L_{5}$,
and let $g^\prime\colon W^\prime\to\mathbb{P}^{3}$
be a~blow up along the lines $L_{1}^\prime,\ldots,L_{5}^\prime$.
Denote by $\widetilde{L}_1^\prime,\ldots,\widetilde{L}_5^\prime$ (respectively, $\widetilde{L}_1,\ldots,\widetilde{L}_5$)
the proper transforms of the lines $L_{1}^\prime,\ldots,L_{5}^\prime$ on the threefold $W$ (respectively, on the threefold $W^\prime$).
Let $h\colon V\to W$ be a~blow up along the curves $\widetilde{L}_1^\prime,\ldots,\widetilde{L}_5^\prime$,
and let $h^\prime\colon V^\prime\to W^\prime$ be a~blow up along the curves $\widetilde{L}_1,\ldots,\widetilde{L}_5$.
Finally, let $\alpha\colon U\to \mathbb{P}^{3}$ be a~blow up of the (singular) curve that is a union of all
lines $L_{1},\ldots,L_{5},L_{1}^\prime,\ldots,L_{5}^\prime$.
Then $U$ has twenty nodes by Lemma~\ref{lemma:double-five}, and there exists a commutative diagram
$$
\xymatrix{
&V\ar@{-->}[rr]^{\tau}\ar@{->}[dl]_{h}\ar@{->}[dr]_{\upsilon}&&V^\prime\ar@{->}[dr]^{h^\prime}\ar@{->}[dl]^{\upsilon^\prime}&\\
W\ar@{->}[drr]_{g}&&U\ar@{->}[d]_{\alpha}&& W^\prime\ar@{->}[dll]^{g^\prime}\\
&&\mathbb{P}^{3}&&}
$$
where $\upsilon$ and $\upsilon^\prime$ are small resolutions of singularities of the threefold $U$,
and $\tau$ is a composition of twenty Atiyah flops.

\begin{remark}
\label{remark:action}
By construction, the action of group $\mathfrak{A}_5$ lifts to  the threefolds $W$, $W^\prime$, $V$, $V^\prime$, and $U$.
Similarly, the action of the group $\mathfrak{S}_5$ lifts to the threefold $U$,
but this action does not lift to $W$ and $W^\prime$.
On the threefolds $V$ and $V^\prime$, the group $\mathfrak{S}_5$ acts biregularly
outside of the curves flopped by $\tau$ and $\tau^{-1}$, respectively.
\end{remark}

Denote by $E_1,\ldots,E_5$ the $g$-exceptional surfaces that are mapped to $L_{1},\ldots,L_{5}$, respectively.
Similarly, denote by $E_1^\prime,\ldots,E_5^\prime$ the $g^\prime$-exceptional surfaces
that are mapped to $L_{1}^\prime,\ldots,L_{5}^\prime$, respectively.
Then all surfaces $E_1,\ldots,E_5,E_1^\prime,\ldots,E_5^\prime$ are
isomorphic to~\mbox{$\mathbb{P}^1\times\mathbb{P}^1$}.

Denote by $\hat{E}_1^\prime,\ldots,\hat{E}_5^\prime$
the $h$-exceptional surfaces that are mapped to the
curves~\mbox{$\widetilde{L}_1^\prime,\ldots,\widetilde{L}_5^\prime$},
respectively.
Similarly, denote by $\check{E}_1,\ldots,\check{E}_5$ the $h^\prime$-exceptional surfaces that are mapped to the curves
$\widetilde{L}_1,\ldots,\widetilde{L}_5$, respectively.
Denote by  $\hat{E}_1,\ldots,\hat{E}_5$ the proper transforms on $V$ of the surfaces $E_1,\ldots,E_5$, respectively.
Finally, denote by  $\check{E}_1^\prime,\ldots,\check{E}_5^\prime$ the proper transforms on $V^\prime$ of the surfaces
$E_1^\prime,\ldots,E_5^\prime$, respectively.
Then $\tau$ maps the surfaces~\mbox{$\hat{E}_1,\ldots,\hat{E}_5$}
to the surfaces $\check{E}_1,\ldots,\check{E}_5$, respectively,
and it maps the surfaces~\mbox{$\hat{E}_1^\prime,\ldots,\hat{E}_5^\prime$}
to the surfaces $\check{E}_1^\prime,\ldots,\check{E}_5^\prime$, respectively.
Denoting by $\Pi$ a plane in $\P^3$, we compute
$$
\left(-K_W\right)^3=\left(4g^*\left(\Pi\right)-\sum\limits_{i=1}^5 E_i\right)^3=
64\left(g^*\left(\Pi\right)\right)^3+12\sum\limits_{i=1}^5 g^*\left(\Pi\right)\cdot E_i^2-
\sum\limits_{i=1}^5 E_i^3=64-60+10=14.
$$
and
$$
\left(-K_V\right)^3=\left(-h^*\left(K_W\right)-
\sum\limits_{i=1}^5\hat{E}_i^\prime\right)^3=
\left(-h^*\left(K_W\right)\right)^3-3\sum\limits_{i=1}^5h^*\left(K_W\right)\cdot\hat{E}_i^{\prime 2}-
\sum\limits_{i=1}^5\hat{E}_i^{\prime 3}=14-10=4.
$$

Denote by $\mathcal{M}_{W}$, $\mathcal{M}_{V}$, $\mathcal{M}_{W^\prime}$, $\mathcal{M}_{V^\prime}$, and $\mathcal{M}_U$
the proper transforms of the linear system~$\mathcal{M}$ on the threefolds $W$, $V$, $W^\prime$, $V^\prime$, and $U$, respectively.
Then it follows from Lemma~\ref{lemma:base-locus-ten-lines} that
$$
\mathcal{M}_{W}\sim -K_W,\quad
\mathcal{M}_{V}\sim -K_V,\quad
\mathcal{M}_{W^\prime}\sim -K_{W^\prime},\quad
\mathcal{M}_{V^\prime}\sim -K_{V^\prime},
$$
and $\mathcal{M}_{U}\sim -K_U$.

\begin{lemma}
\label{lemma:base-locus-ten-lines-blow-up-W}
The base locus of the linear system $\mathcal{M}_{W}$  does not contain curves that are different from the
curves~\mbox{$\widetilde{L}_1^\prime,\ldots,\widetilde{L}_5^\prime$}.
Similarly, the base locus of $\mathcal{M}_{W^\prime}$  does not contain curves that are different from the curves $\widetilde{L}_1,\ldots,\widetilde{L}_5$.
\end{lemma}

\begin{proof}
It is enough to prove the first assertion of the lemma.
Suppose that the base locus of the linear system $\mathcal{M}_{W}$ contains an irreducible curve $Z$
that is different from the
curves~\mbox{$\widetilde{L}_1^\prime,\ldots,\widetilde{L}_5^\prime$}.
Then $Z$ is contained in one of the surfaces $E_1,\ldots,E_5$ by Lemma~\ref{lemma:base-locus-ten-lines}.

By Lemma~\ref{lemma:base-locus-ten-lines}, the curve $Z$ is a fiber of some of the natural projections $E_i\to L_i$,
because otherwise two general surfaces in $\mathcal{M}_W$ would
be tangent in a general point of~$L_i$.
In particular, the only curves in the base locus of the linear system $\mathcal{M}_{W}$ are
$\widetilde{L}_i^\prime$ and possibly some fibers of the projections $E_i\to L_i$.
This shows that $-K_W$ is nef.
Indeed,~\mbox{$-K_W$} has positive intersections with the fibers of the projections $E_i\to L_i$,
it has trivial intersection with all curves $\widetilde{L}_1^\prime,\ldots,\widetilde{L}_5^\prime$,
and $-K_W\sim\mathcal{M}_W$ has non-negative intersection with any other curve.

Let $Z_1=Z,Z_2,\ldots,Z_r$ be the $\mathfrak{A}_5$-orbit of the curve $Z$.
Then $r\geqslant 20$ by Corollary~\ref{corollary:five-lines-orbits}.
Pick two general surfaces $M_1$ and $M_2$ in the linear system $\mathcal{M}_{W}$.
By Lemma~\ref{lemma:base-locus-ten-lines}, one has
$$
M_1\cdot M_2=\sum_{i=1}^5\widetilde{L}_i^\prime+m\sum_{i=1}^rZ_i+\Delta
$$
for some positive integer $m$ and some effective one-cycle $\Delta$ whose support
contains none of the curves $\widetilde{L}_1^\prime,\ldots,\widetilde{L}_5^\prime$ and $Z_1,\ldots,Z_r$.
Hence
\begin{multline*}
14=-K_W^3=-K_W\cdot M_1\cdot M_2=-K_W\cdot\Big(\sum_{i=1}^5\widetilde{L}_i^\prime+m\sum_{i=1}^rZ_i+\Delta\Big)=\\
=-5K_W\cdot\widetilde{L}_1^\prime-mrK_W\cdot Z-K_W\cdot\Delta=mr-K_W\cdot\Delta\geqslant mr\geqslant r\geqslant 20,
\end{multline*}
which is absurd.
\end{proof}

\begin{lemma}
\label{lemma:base-locus-ten-lines-blow-up}
The linear system  $\mathcal{M}_{V}$ is base point free.
\end{lemma}

\begin{proof}
It is enough to show that $\mathcal{M}_{V}$ is free from base curves.
Indeed, if the base locus of the linear system $\mathcal{M}_{V}$ does not contain base curves,
then $\mathcal{M}_{V}$ cannot have more than~\mbox{$-K_V^3=4$} base points, because $\mathcal{M}_{V}\sim -K_{V}$.
On the other hand, $V$ does not contain $\mathfrak{S}_5$-orbits of length less than $12$,
because there are no $\mathfrak{S}_5$-orbits of such length on $\mathbb{P}^3$ by Lemma~\ref{lemma:P3-S5-orbits}.

Suppose that the base locus of the linear system $\mathcal{M}_{V}$ contains an irreducible curve~$Z$.
If~$Z$ is not contained in any of the surfaces $\hat{E}_1^\prime,\ldots,\hat{E}_5^\prime$,
then the curve $h(Z)$ is a base curve of the linear system
$\mathcal{M}_{W}$ and $h(Z)$ is different from the
curves~\mbox{$\widetilde{L}_1^\prime,\ldots,\widetilde{L}_5^\prime$}.
This is impossible by Lemma~\ref{lemma:base-locus-ten-lines-blow-up-W}.
Similarly, if $Z$ is not contained in any of the surfaces~\mbox{$\hat{E}_1,\ldots,\hat{E}_5$},
then the curve $h^\prime\circ\tau(Z)$ is a base curve of the linear system
$\mathcal{M}_{W^\prime}$ that is different from the
curves~\mbox{$\widetilde{L}_1,\ldots,\widetilde{L}_5$}.
This is again impossible by Lemma~\ref{lemma:base-locus-ten-lines-blow-up-W}.
Thus, $Z$ is contained in one of the surfaces $\hat{E}_1,\ldots,\hat{E}_5$,
and in one of the surfaces $\hat{E}_1^\prime,\ldots,\hat{E}_5^\prime$.
In particular, the curves flopped by $\tau$ are not contained in the base locus of $\mathcal{M}_V$.

Without loss of generality, we may assume that $Z=\hat{E}_1\cap\hat{E}_2^\prime$.
Let $C$ be the curve flopped by $\tau$ that is contained in $\hat{E}_1$ and intersects $\hat{E}_2^\prime$.
Then $C$ intersects $Z$ by one point.
On the other hand, we have $-K_V\cdot C=0$.
Since~\mbox{$\mathcal{M}_{V}\sim -K_V$},
this implies that $C$ is disjoint from a general surface in $\mathcal{M}_{V}$.
This is impossible, because $C\cap Z\ne\varnothing$,
while $Z$ is contained in the base locus of the linear system $\mathcal{M}_{V}$.
\end{proof}

\begin{corollary}
\label{corollary:base-locus-ten-lines-blow-up}
The linear systems $\mathcal{M}_{V^\prime}$, and $\mathcal{M}_U$ are also base point free.
\end{corollary}

\begin{proof}
Recall that $\mathcal{M}_{V}\sim -K_V$.
Thus, the general surface of $\mathcal{M}_{V}$ is disjoint from all curves flopped by $\tau$,
because $\mathcal{M}_{V}$ is base point free by Lemma~\ref{lemma:base-locus-ten-lines-blow-up}.
\end{proof}

\begin{corollary}
\label{corollary:base-locus-ten-lines}
The base locus of $\mathcal{M}$ consists of the lines $L_1,\ldots,L_5,L_1^\prime,\ldots,L_5^\prime$.
\end{corollary}

By Lemma~\ref{lemma:base-locus-ten-lines-blow-up} and Corollary~\ref{corollary:base-locus-ten-lines-blow-up}, the divisors $-K_{V}$, $-K_{V^\prime}$, and $-K_U$ are nef.
Since
$$
-K_{V}^3=-K_{V^\prime}^3=-K_U^3=4,
$$
these divisors are also big.
Thus, the Kawamata--Viehweg vanishing theorem and the Riemann--Roch formula give
\begin{equation}\label{eq:S5-h0}
h^0\Big(\mathcal{O}_V\big(-K_V\big)\Big)=h^0\Big(\mathcal{O}_{V^\prime}\big(-K_{V^\prime}\big)\Big)=h^0\Big(\mathcal{O}_U\big(-K_U\big)\Big)=4.
\end{equation}
In particular, one has $|-K_V|=\mathcal{M}_V$, $|-K_V^\prime|=\mathcal{M}_{V^\prime}$, and $|-K_U|=\mathcal{M}_U$.

\begin{lemma}
\label{lemma:anticanonical-linear-system-W5-ten-lines}
The $\mathfrak{S}_5$-representation $H^0(\mathcal{O}_U(-K_U))$
is irreducible.
\end{lemma}

\begin{proof}
By Lemma~\ref{lemma:no-quartic-through-ten-lines},
there are no $\mathfrak{S}_5$-invariant quartic surfaces in $\mathbb{P}^3$
that pass through the ten lines $L_1,\ldots,L_5,L_1^\prime,\ldots,L_5^\prime$.
This implies that $H^0(\mathcal{O}_U(-K_U))$ does not contain one-dimensional subrepresentations.
Hence it is irreducible by Remark~\ref{remark:pohuj}.
\end{proof}

Lemma~\ref{lemma:base-locus-ten-lines-blow-up} together with~\eqref{eq:S5-h0}
implies that there is an~$\mathfrak{S}_5$-equivariant commutative diagram
\begin{equation}
\label{equation:commutativ-diagram}
\xymatrix{
&U\ar@{->}[dl]_{\alpha}\ar@{->}[dr]^{\beta}&\\
\mathbb{P}^{3}\ar@{-->}[rr]_{\phi}&&\mathbb{P}^4}
\end{equation}
where $\phi$ is the rational map given by $\mathcal{M}$,
and $\beta$ is a morphism given by the anticanonical linear system $|-K_U|$.
By Lemma~\ref{lemma:anticanonical-linear-system-W5-ten-lines}
the projective space $\mathbb{P}^4$ in~\eqref{equation:commutativ-diagram} is a projectivization
of an irreducible $\mathfrak{S}_5$-representation.

For $1\leqslant i<j\leqslant 5$, let $\Lambda_{ij}$ be the intersection line of
the plane spanned by~$L_i$ and~$L_j^\prime$ with the plane spanned by $L_i^\prime$ and $L_j$.
Note that the stabilizer of $\Lambda_{ij}$ in $\mathfrak{S}_5$ contains
a subgroup isomorphic to~$\mathrm{D}_{12}$.
Actually, this implies that the stabilizer of
$\Lambda_{ij}$ in $\mathfrak{S}_5$ is isomorphic to~$\mathrm{D}_{12}$, since
$\mathrm{D}_{12}$ is a maximal proper subgroup in $\mathfrak{S}_5$
(see Remark~\ref{remark:S5-subgroups}) and there are no $\mathfrak{S}_5$-invariant
lines in~$\mathbb{P}^3$ by Corollary~\ref{corollary:characters}(i).
Denote by $\hat{\Lambda}_{ij}$ the proper transform of the line $\Lambda_{ij}$ on the threefold $V$,
and denote by $\overline{\Lambda}_{ij}$ its proper transform on $U$. Then
$$
-K_V\cdot \hat{\Lambda}_{ij}=0.
$$
Since $\upsilon$ is a small birational morphism, we also obtain $-K_U\cdot \overline{\Lambda}_{ij}=0$.

We see that each curve $\overline{\Lambda}_{ij}$ is contracted by $\beta$ to a point.
Note that the stabilizer of $\Lambda_{ij}$ in~$\mathfrak{S}_5$ is isomorphic to $\mathrm{D}_{12}$.
Since $-K_U^3=4$, the image of $\beta$ is either an $\mathfrak{S}_5$-invariant
quartic threefold or an \mbox{$\mathfrak{S}_5$-invariant} quadric threefold.
Applying Corollary~\ref{corollary:Q-orbits} together with Lemma~\ref{lemma:P4-S5-orbits}, we obtain the following.

\begin{corollary}
\label{corollary:X-1-6-Q}
The morphism $\beta$ is birational on its image,
and its image is a quartic threefold.
\end{corollary}

Now Lemmas~\ref{lemma:P4-S5-orbits} and~\ref{lemma:X-1-6} imply that the image of $\beta$ is the quartic $X_{\frac{1}{6}}$.
This proves

\begin{corollary}
\label{corollary:X-1-6-rational}
The threefold $X_{\frac{1}{6}}$ is rational.
\end{corollary}

Ten curves $\overline{\Lambda}_{ij}$ are mapped by $\gamma$ to
ten singular points of the threefold $X_{\frac{1}{6}}$.
Twenty singular points of $U$ are mapped by $\gamma$ to another twenty singular points of $X_{\frac{1}{6}}$.
Let us describe the curves in $U$ that are contracted by $\gamma$ to the remaining ten singular points of the threefold $X_{\frac{1}{6}}$.
To do this, we need

\begin{lemma}
\label{lemma:tangent-quadric}
Let $\ell_1$, $\ell_2$, $\ell_3$ and $\ell_4$ be pairwise skew lines in $\mathbb{P}^3$.
Suppose that there is a unique line $\ell\subset\mathbb{P}^3$
that intersects $\ell_1$, $\ell_2$, $\ell_3$ and $\ell_4$.
Let $\pi\colon Y\to\mathbb{P}^3$ be a blow up of the line $\ell$, and
$E\cong \mathbb{P}^1\times\mathbb{P}^1$ be the exceptional divisor of $\pi$.
Denote by $\tilde{\ell}_1$, $\tilde{\ell}_2$, $\tilde{\ell}_3$ and $\tilde{\ell}_4$
the proper transforms on $Y$ of the lines $\ell_1$, $\ell_2$, $\ell_3$ and $\ell_4$, respectively.
Then there exists a unique curve $C\subset E$ of bi-degree $(1,1)$ that intersects
the curves $\tilde{\ell}_1$, $\tilde{\ell}_2$, $\tilde{\ell}_3$ and $\tilde{\ell}_4$.
\end{lemma}

\begin{proof}
The lines $\ell_1$, $\ell_2$, and $\ell_3$ are contained in a unique quadric surface $S\subset\mathbb{P}^3$.
Note that $S$ is smooth, because $\ell_1$, $\ell_2$, and $\ell_3$ are disjoint.
Furthermore, the line $\ell$ is contained in $S$, because $\ell$ intersects the lines $\ell_1$, $\ell_2$, and $\ell_3$ by assumption.
Moreover, the line $\ell_4$ is tangent to $S$, since otherwise there would be either two or infinitely many lines
in $\mathbb{P}^3$ that intersect $\ell_1$, $\ell_2$, $\ell_3$ and $\ell_4$.
Denote by $\tilde{S}$ the proper transform on $Y$ of the quadric
surface~$S$.
Then $\tilde{S}$ contains the curves $\tilde{\ell}_1$, $\tilde{\ell}_2$, and $\tilde{\ell}_3$.
Moreover, $\tilde{S}$ intersects the curve $\tilde{\ell}_4$.
Thus~\mbox{$\tilde{S}\vert_{E}$} is the required curve $C$.
\end{proof}

By Lemmas~\ref{lemma:double-five} and~\ref{lemma:tangent-quadric},
each surface $E_i\cong\mathbb{P}^1\times\mathbb{P}^1$ contains a unique smooth rational curve $C_i$  of bi-degree $(1,1)$
that passes through all four points of the intersection of $E_i$ with the curves $\widetilde{L}_1^\prime,\ldots,\widetilde{L}_5^\prime$
(recall that $E_i\cap\widetilde{L}_i^\prime=\varnothing$).
Similarly, each surface $E_i^\prime\cong\mathbb{P}^1\times\mathbb{P}^1$ contains a unique smooth rational curve $C_i^\prime$  of bi-degree $(1,1)$
that passes through all four points of the intersection of $E_i^\prime$ with the curves $\widetilde{L}_1,\ldots,\widetilde{L}_5$.
Denote by $\hat{C}_1,\ldots,\hat{C}_5$  the proper transforms on the threefold $V$ of the curves $C_1,\ldots,C_5$, respectively.
Similarly, denote by $\check{C}_1^\prime,\ldots,\check{C}_5^\prime$  the proper transforms
on the threefold $V^\prime$ of the curves $C_1^\prime,\ldots,C_5^\prime$, respectively.
Then
$$
-K_{V}\cdot\hat{C}_i=-K_{V^\prime}\cdot\check{C}_i^\prime=0.
$$
This implies that the proper transforms of the curves $\hat{C}_1,\ldots,\hat{C}_5$
on the threefold $V^\prime$ are $(-2)$-curves on the surfaces
$\check{E}_1,\ldots,\check{E}_5$, respectively.
Similarly, the proper transforms of the curves $\check{C}_1^\prime,\ldots,\check{C}_5^\prime$ on the threefold $V$
are $(-2)$-curves on the surfaces $\hat{E}_1^\prime,\ldots,\hat{E}_5^\prime$, respectively.
Thus, all surfaces $\hat{E}_1^\prime,\ldots,\hat{E}_5^\prime,\check{E}_1,\ldots,\check{E}_5$ are isomorphic to the Hirzebruch surface~$\mathbb{F}_2$.

Denote by $\overline{C}_1,\ldots,\overline{C}_5,\overline{C}_1^\prime,\ldots,\overline{C}_5^\prime$ the images of the curves
$\hat{C}_1,\ldots,\hat{C}_5,\check{C}_1^\prime,\ldots,\check{C}_5^\prime$ on the threefold $U$, respectively.
Then
$$
-K_{U}\cdot\overline{C}_i=-K_{U}\cdot\overline{C}_i^\prime=0,
$$
because $-K_{V}\cdot\hat{C}_i=-K_{V^\prime}\cdot\check{C}_i^\prime=0$,
and $\upsilon$ and $\upsilon^\prime$ are small birational morphisms.
Thus, the ten
curves~\mbox{$\overline{C}_1,\ldots,\overline{C}_5,\overline{C}_1^\prime,\ldots,\overline{C}_5^\prime$}
are contracted by the morphism $\beta$ to ten singular points of~$X_{\frac{1}{6}}$.

It would be interesting to extend the commutative
diagram~\eqref{equation:commutativ-diagram}
to an $\mathfrak{S}_5$-Sarkisov link similar to
the $\mathfrak{A}_6$-Sarkisov link~\eqref{equation:A6-link}.

\end{document}